\documentclass[10pt]{article}
\usepackage{bbm}
\usepackage{amsfonts}
\usepackage{amsmath}
\allowdisplaybreaks
\usepackage{amssymb}
\usepackage{fancyhdr}
\usepackage{setspace}
\usepackage{latexsym}
\usepackage{mathrsfs}
\usepackage{amsthm}

\usepackage{cite,enumitem,graphicx}
\usepackage[colorlinks=true,urlcolor=blue,
citecolor=red,linkcolor=blue,linktocpage,pdfpagelabels,
bookmarksnumbered,bookmarksopen]{hyperref}
\usepackage[english]{babel}

\numberwithin{equation}{section}
\newtheorem{theorem}{Theorem}[section]
\newtheorem{proposition}[theorem]{Proposition}
\newtheorem{lemma}[theorem]{Lemma}

\newtheorem{remark}[theorem]{Remark}

\newcommand{\ud}{\mathrm{d}}

\makeatletter 
\renewcommand\theequation{\thesection.\arabic{equation}}
\@addtoreset{equation}{section}
\makeatother
\numberwithin{equation}{section}

\begin{document}

\begin{center}
\textbf{Normalized solutions for nonlinear
Schr\"{o}dinger equation involving potential and Sobolev critical exponent}\\
\end{center}

\begin{center}
 Zhen-Feng Jin$^{1}$, Weimin Zhang$^{2, *}$\\

$^{1}$ School of Mathematics and Computer Science, Shanxi Normal University,\\ Taiyuan, 030031, P.R. China\\
$^{2}$ School of Mathematical Sciences,  Key Laboratory of Mathematics and Engineering Applications (Ministry of Education) \& Shanghai Key Laboratory of PMMP,  East China Normal University, Shanghai, 200241, P.R. China
\end{center}

\begin{center}
\renewcommand{\theequation}{\arabic{section}.\arabic{equation}}
\numberwithin{equation}{section}
\footnote[0]{\hspace*{-7.4mm}
$^*$Corresponding author.
{E-mail addresses: jinzhenfeng@sxnu.edu.cn (Z.-F. Jin), zhangweimin2021@gmail.com (W. Zhang).}}
\end{center}
\begin{abstract}
In this paper, we consider the existence of positive solutions with prescribed $L^2$-norm for the following nonlinear Schr\"{o}dinger equation involving potential and Sobolev critical exponent
\begin{equation*}
\begin{cases}
-\Delta u+V(x)u=\lambda u+\mu |u|^{p-2}u+|u|^{\frac{4}{N-2}}u \;\;\text { in } \mathbb{R}^N, \\
\|u\|_2=a>0,\\
\end{cases}
\end{equation*}
where $N\ge 3$, $\mu>0$, $p\in [2+\frac{4}{N}, \frac{2N}{N-2})$ and $V\in C^1(\mathbb{R}^N)$. Under different assumptions on $V$, we derive two different Pohozaev identities. Based on these two cases, we respectively obtain the existence of positive solution. As far as we are aware, we did not find any works on normalized solutions with Sobolev critical growth and potential $V \not\equiv 0$. Our results extend some results of Wei and Wu [J. Funct. Anal. 283(2022)] to the potential case.
\end{abstract}


\section{Introduction}
\noindent

This paper is concerned with the existence of positive solutions with prescribed $L^2$-norm for the following nonlinear Schr\"{o}dinger equation
\begin{equation}\label{041003}
\begin{cases}
-\Delta u+V(x)u=\lambda u+\mu |u|^{p-2}u+|u|^{2^{*}-2}u \;\;\text { in } \mathbb{R}^N, \\
\|u\|_2=a>0,\\
\end{cases}
\end{equation}
where $N\ge 3$, $\mu>0$, $p\in [2+\frac{4}{N},2^{*})$ with
$2^{*}:=\frac{2N}{N-2}$, and
$$\|u\|^2_2=\int_{\mathbb{R}^N} u^2 \ud x.$$
 The parameter $\lambda\in\mathbb{R}$ arises as a Lagrange multiplier with respect to the mass constraint $\|u\|_{2}=a$. Here, $V : \mathbb{R}^N\to \mathbb{R}$ is a potential function.\par
 Problem \eqref{041003} comes from the study of standing waves for the nonlinear Schr\"{o}dinger equation
\begin{equation}\label{2207181059}
    iw_t-\Delta w + V(x)w=f(w) \;\; \text{in}~ \mathbb{R}^N\times (0,\infty),
\end{equation}
where $w$ has the form
\begin{equation}\label{2207181100}
w(x,t)=e^{-i\lambda t} u(x),\;\; (x,t)\in \mathbb{R}^N\times (0,\infty),
\end{equation}
and $u$ is a real function. For problem \eqref{2207181059}-\eqref{2207181100}, the $L^2$-norm of $u$ stands for the mass of a particle. If $f(w)=e^{-i\lambda t}f(u)$, \eqref{2207181059}-\eqref{2207181100} will be reduced to
\begin{equation}\label{22072316}
 -\Delta u+V(x)u=\lambda u+f(u) \;\;\text { in } \mathbb{R}^N.
\end{equation}
Such problem for given $\lambda$ is called the {\it fixed frequency problem}, which has been widely studied for the decades. A huge literature exists, we will not intend to summarize here, the interested readers can refer to \cite{Ambrosetti2006, Cerami2006, Struwe2008} and references therein. Here we are concerned with the case of  prescribed mass, that is, for some fixed $a>0$, we try to find solutions of \eqref{22072316} on the manifold
\begin{equation}\label{2303012154}
S_a:=\{u\in H^1(\mathbb{R}^N):\|u\|_{2}=a\},
\end{equation}
with $\lambda\in\mathbb{R}$. Commonly, such solutions are called normalized solutions. A natural approach is applying variational method to \eqref{22072316}-\eqref{2303012154}. One can derive solutions to \eqref{22072316}-\eqref{2303012154} by looking for critical points of the associated energy functional on $S_a$,
\begin{equation}\label{2303021513}
\begin{aligned}
E_{V}(u)=\frac{1}{2}\|\nabla u\|_2^2 +\frac12 \int_{\mathbb{R}^N}V(x)u^2\ud x-\int_{\mathbb{R}^N}F(u) \ud x
\end{aligned}
\end{equation}
where
$$F(u) =\int_0^{u}f(t)\ud t.$$
\subsection{Non-potential case $V\equiv 0$.} The simplest example is the power nonlinearities $f(u)=|u|^{p-2}u$. Consider
\begin{equation}\label{22071813}
    -\Delta u+u= |u|^{p-2}u~\text { in } \mathbb{R}^N, \;\; u\in H^1(\mathbb{R}^N).
\end{equation}
It is well-known that for $p\in (2,2^*)$, \eqref{22071813} has a unique positive solution up to a translation, which can be chosen to be radial (see \cite{Kong1989}). If we denote the radial solution by $U_p$, one can check that
\[
U_{\lambda, p}(x):=\lambda^{\frac{1}{p-2}}U_p({\sqrt{\lambda}}x),\;\; \lambda>0,
\]
is the unique positive solution (up to a translation) to
\begin{equation}\label{2301112026}
-\Delta u+\lambda u= |u|^{p-2}u\;\;\text { in } \mathbb{R}^N, \;\; u\in H^1(\mathbb{R}^N).
\end{equation}
A direct computation shows that
\[
\|U_{\lambda, p}\|_2^2=\lambda^{\frac{4-(p-2)N}{2(p-2)}}\|U_p\|_2^2.
\]
If $p\neq \hat{p}:=2+\frac{4}{N}$, for any $a>0$, there exists a unique $\lambda>0$ such that $\|U_{\lambda, p}\|_2=a$, which means that \eqref{2301112026}
has a unique positive radial normalized solution $u\in S_a$ for any $p\in (2, \hat{p})\cup (\hat{p}, 2^*)$.
While for $p=\hat{p}$, the equation \eqref{2301112026} has a positive normalized solution $u\in S_a$ if and only if $a=\|U_p\|_2$. Therefore, $\hat{p}=2+\frac{4}{N}$ is called the mass critical exponent or $L^2$-critical exponent to \eqref{2301112026}.

Moreover, the mass critical exponent plays an important role in the geometry of energy functional. In order to preserve the $L^2$-norm, usually we use the scaling
\begin{equation}\label{2301111617}
 u_h(x):=h^{\frac{N}{2}}u(h x),\;\; h>0,
\end{equation}
which gives readily
\[
\|\nabla u_h\|_2^2=h^2\|\nabla u\|_2^2, \;\;\| u_h\|_p^p=h^{\frac{N(p-2)}{2}}\|u\|_p^p.
\]
Remark that to have $h^2=h^{\frac{N(p-2)}{2}}$ for $h\neq 1$, one needs $p=\hat{p}$. Thus, for fixed $u\in S_a$, the sign of $p-\hat{p}$ decides the shape of $h\mapsto E_0(u_h)$, seeing the Gagliardo-Nirenberg inequality (see \cite{N1959}): {\it For any $N\ge 3$ and $p\in [2,2^*]$, there exists a constant $C_{N,\, p}$ depending on $N$, $p$ such that
\begin{equation}\label{22071819}
\|u\|_p\le C_{N,\, p}\|u\|_2^{1-\gamma_p}\|\nabla u\|_2^{\gamma_p},\;\; \forall~ u\in H^{1}(\mathbb{R}^N)
\end{equation}
where $\gamma_p=\frac{N(p-2)}{2p}$. The inequality holds true also in $N\le 2$ for any $2\le p<\infty$.}\par
We can observe that
\[
p\gamma_p
\begin{cases}
<2\;\; \text{if}~2<p<\hat{p}\\
=2\;\; \text{if}~p=\hat{p}\\
>2\;\; \text{if}~\hat{p}<p<2^*,\\
\end{cases}\;\; \text{and}\;\; \gamma_{2^*}=1.
\]
The sign of $p-\hat{p}$ decides whether the functional $E_0$ is bounded from below on $S_a$.
For $p=2^*$, we have a particular case of \eqref{22071819}, the Sobolev inequality
\begin{equation}\label{2301112226}
\mathcal{S}\|u\|_{2^*}^2\le \|\nabla u\|_2^2, \;\; \forall~ u\in D^{1, 2}(\mathbb{R}^N),~N\ge 3.
\end{equation}
The sharp constant $\mathcal{S}$ for \eqref{2301112226} is called the best Sobolev constant. It is well known from \cite{Talenti1976} that
\[
\widetilde{U}_{\varepsilon, y}(x)=[N(N-2)]^{\frac{N-2}{4}}\left( \frac{\varepsilon}{\varepsilon^2+|x-y|^2} \right)^{\frac{N-2}{2}},\;\; \varepsilon>0, \,y\in \mathbb{R}^N,
\]
are the unique functions to achieve \eqref{2301112226}, which are also the only classical solutions (see \cite{CGS1989}) to
\[
-\Delta u= u^{2^*-1},\;\; u>0\;\; \text{in}~\mathbb{R}^N.
\]
Clearly, $\widetilde{U}_{\varepsilon, y}\in L^2(\mathbb{R}^N)$ if and only if $N\ge 5$. Thus, for the critical case $p=2^*$, \eqref{2301112026} admits a positive radial normalized solution $\widetilde{U}_{\varepsilon, 0}\in S_a$ with $\lambda=0$ and a  unique choice of $\varepsilon>0$ if $N\ge5$, and no positive solution exists any more if $N\le 4$. \par
For nonhomogeneous nonlinearities, the scaling method does not work. Jeanjean  \cite{Jeanjean97} did a seminal work and considered a class of mass supercritical and Sobolev subcritical problem where $f(u)$ can be chosen as $\sum_{1\leq j \leq k}a_j|u|^{\sigma_j-2}u$ with $k\ge 1, a_j>0$ and
\[
\hat{p}<\sigma_j<2^* \;\;\text{if}~N\ge 3;\;\;\hat{p}<\sigma_j \;\;\text{if}~N=1, 2.
\]
He showed that the energy functional $E_0$ possesses a mountain pass geometry on $S_a$. A crucial step in \cite{Jeanjean97} is to construct an augmented functional
$$\widetilde{E_0}(u, h):=E_0(u_h), \;\; (u, h)\in S_a\times \mathbb{R}^+,$$
in order to obtain a Palais-Smale sequence approaching the Pohozaev manifold
\[
\mathcal{P}_0:=\left\{u\in S_a: \frac{\ud}{\ud h}E_0(u_h)\Big|_{h=1}=0 \right\}.
\]
For more references on normalized solution problem with Sobolev subcritical growth, we refer to \cite{Soave20_JDE, JL2020} and references therein. A lot of works have been done following Jeanjean's method, see for instance  \cite{BMRV2021, MRV2022, JL2020, Soave20_JFA, Soave20_JDE}. Here we are mainly concerned with the Sobolev critical growth problem.

Recently, Soave \cite{Soave20_JFA} considered the combined nonlinearities involving Sobolev critical exponent
\begin{equation}\label{2207191018}
\begin{cases}
-\Delta u=\lambda u+ \mu|u|^{p-2}u+ |u|^{2^*-2}u \;\;\text { in } \mathbb{R}^N, \\
\|u\|_{2}=a>0,\\
\end{cases}
\end{equation}
where $N\ge 3$ and $\mu\in \mathbb{R}$. The corresponding energy functional to \eqref{2207191018} is
\begin{equation*}
\begin{aligned}
J_{\mu}(u):=\frac{1}{2}\|\nabla u\|_2^2
-\frac{\mu}{p}\|u\|_p^{p}
-\frac{1}{2^{*}}\|u\|_{2^*}^{2^{*}},\;\; u\in S_a.
\end{aligned}
\end{equation*}
Recall that a ground state of \eqref{2207191018} is a solution having the minimal energy among all solutions. The ground state of \eqref{2207191018} can be obtained by finding the minimizers of $J_{\mu}$ on the
Pohozaev manifold
\[
\mathcal{P}_{\mu}:=\{u\in S_a: \|\nabla u\|_2^2-\mu \gamma_p\|u\|_p^p-\|u\|_{2^*}^{2^*}=0 \}.
\]
Let
\begin{equation}\label{22071824}
m_{a}:= \underset{u\in \mathcal{P}_{\mu}}{\inf}{J_{\mu}(u)}.
\end{equation}
Soave investigated the existence of ground state solutions to \eqref{2207191018} for different exponent $p$.
He proved that for $p\in (2, 2^*)$, \eqref{22071824} can be attained provided $N\ge 3$, $\mu$, $a>0$ and
\begin{equation}\label{2312232020}
 \mu a^{{p(1-\gamma_p)}}<\alpha_{N,\, p}
\end{equation}
for some $\alpha_{N,\, p}>0$, see \cite[Theorem 1.1]{Soave20_JFA}. Also, he proposed some natural questions, including
\begin{itemize}
\item[($Q_1$)]: Does $J_{\mu}|_{S_a}$ have a critical point of mountain-pass type in the case $p\in (2, \hat{p})$?
\item[($Q_2$)]: Does $J_{\mu}|_{S_a}$ have a ground state if $\mu>0$ and $ \mu a^{{p(1-\gamma_p)}}$ is large?
\end{itemize}
Jeanjean and Le \cite{JL2022} and Wei and Wu \cite{WW22} answered question $(Q_1)$ for $N\ge 4$ and $N=3$ respectively. Under the condition \eqref{2312232020}, the authors in \cite{JL2022, WW22} showed that $J_{\mu}$ has a mountain pass geometry around the ground state, and the mountain pass level $M_a$ can be upper bounded by
\begin{equation}\label{2302151613}
M_a<m_a+\frac{1}{N}\mathcal{S}^{\frac{N}{2}}.
\end{equation}
It was verified in \cite[Proposition 1.11]{JL2022} that for $N\ge 3$, if the mountain pass level $M_a$ satisfies \eqref{2302151613}, then any radial Palais-Smale sequence with level $M_a$ approaching the Pohozaev manifold $P_{\mu}$ is relatively compact in $H^1(\mathbb{R}^N)$. Moreover, Wei and Wu also answered $(Q_2)$ in \cite{WW22} for $p\in [\hat{p}, 2^*)$. Now we summarize the existence results of \cite{Soave20_JFA, WW22} for \eqref{2207191018} with $p\in [\hat{p}, 2^*)$:
\begin{theorem}\label{2207231547}
Let $N\ge 3$, $\hat{p}\le p<2^*$ and let $\mu, a>0$.
\begin{itemize}
\item[$(i)$] If $p=\hat{p}$, then for
$$\mu a^{\frac4N}< \alpha:= \frac{\hat{p}}{2(C_{N,\hat{p}})^{\hat{p}}},$$
$m_a$ can be attained by some $u$ which is positive, radially symmetric, and a solution to \eqref{2207191018} for some $\lambda<0$. Here $C_{N, \hat{p}}$ is the best constant for Gagliardo-Nirenberg inequality given by \eqref{22071819}. Moreover, $m_a$ can not be attained if $\mu a^{\frac4N}\ge \alpha$.
\item[$(ii)$] If $\hat{p}<p<2^*$, then for any $\mu, a>0$, $m_{a}$ is attained by some $u$ which is positive, radially symmetric, and a solution to \eqref{2207191018} with $\lambda<0$.
\end{itemize}
\end{theorem}
\begin{remark}\label{2302021620}
When $p=\hat{p}$, we can get the precise value of $\alpha$ owing to \cite[(5.1)]{Soave20_JFA} and \cite[(3.13)]{WW22}.
\end{remark}

\subsection{Potential case $V\not\equiv 0$.}
When $V \not\equiv 0$, many works were done with variational method to find solutions of \eqref{22072316}-\eqref{2303012154}.
\subsubsection{Mass-subcritical growth}
When $f$ has mass-subcritical growth, for example, $f(u)=\sum_{1\leq j\leq k}a_j|u|^{\sigma_j-2}u$ with $k\ge 1, a_j>0$, $N\geq 1$ and
\[
2<\sigma_j<\hat{p},
\]
the associated energy functional $E_V$ (see \eqref{2303021513}) is bounded from below on $S_a$. In this case, to get a minimizer, the main difficulty is to prove that the minimum level is sub-additive with respect to the mass. For example:

\begin{itemize}
\item In \cite{IM2020}, Ikoma and Miyamoto assumed that
\[
V\in C(\mathbb{R}^N),\;\; \underset{|x|\to \infty}{\lim}V(x)=\underset{x\in \mathbb{R}^N}{\sup}V(x)=0;
\]
if $N\ge 5$, the following additional assumption is imposed
\[
 V\in C^{0, 1}(\mathbb{R}^N),\;\; \nabla V(x)\cdot x\le \frac{(N-2)^2}{2|x|^2} ~\text{for}~a.e.~x\in \mathbb{R}^N\backslash \{0\}.
\]
They applied the concentration-compactness arguments of Lions \cite{Lions84-1, Lions84-2}.
\item Zhong and Zou \cite{ZZ2021} considered $V\in C(\mathbb{R}^N)$ satisfying
\[
\underset{|x|\to \infty}{\lim}V(x)=\underset{x\in \mathbb{R}^N}{\sup}V(x)\in (0, \infty],\;\; V(0)=\underset{x\in \mathbb{R}^N}{\min}V(x).
\]
They presented a new approach based on iteration to obtain the strictly sub-additive inequality.
\item Alves and Ji \cite{CJ2022} considered $f(u)=|u|^{p-2}u$ with $p\in (2, \hat{p})$, and a positive, ${\mathbb Z}^N$-periodic (or asymptotically periodic) potential $V$. In their proof, ${\mathbb Z}^N$-periodicity prevents the vanishing of the minimizing sequence at infinity.
\end{itemize}

\subsubsection{Mass-supercritical and Sobolev subcritical growth} When $f$ has mass-supercritical and Sobolev subcritical growth, for example $f(u)=\sum_{1\leq j \leq k}a_j|u|^{\sigma_j-2}u$ with $k\ge 1, a_j>0$ and
\[
\hat{p}<\sigma_j<2^* \;\;\text{if}~N\ge 3;\;\;\hat{p}<\sigma_j\;\;\text{if}~N=1, 2,
\]
the energy functional is unbounded from below on $S_a$. We often use the mountain pass geometry to handle the existence issue of solutions to \eqref{22072316}-\eqref{2303012154}.

\begin{itemize}
\item Molle {\it et al.} \cite{MRV2022} investigated the case $f(u)=|u|^{p-2}u$ with $p\in (\hat{p}, 2^*)$. They constructed a splitting lemma to obtain the compactness of Palais-Smale sequence. If moreover
$$\max\{\|V\|_{N/2}, \left\|V(x)|x|\right\|_N\}<L,\;\; V(x)\le 0,$$
for some $L=L(N, p)>0$, they derived a mountain-pass solution at a positive level.
\item Bartsch {\it et al.} \cite{BMRV2021} discussed the non-trapping potential case, with
\[
V(x)\ge \underset{|x|\to +\infty}{\lim\inf}\,V(x)>-\infty,
\]
where the mountain pass structure by Jeanjean is destroyed. They constructed a linking geometry developed by \cite{BC1990}, and using a minimax argument, they obtained the existence of solutions with high Morse index.
\item In \cite{DZ2022}, Ding and Zhong treated the case for more general $f$ with mass super-critical and Sobolev subcritical growth, satisfying some Ambrosetti-Rabinowitz type condition. They assumed that $V$ is negative, twice differentiable a.e. in $\mathbb{R}^N$ and required some compactness conditions or Poincar\'e inequality on $V$, which yields that the Pohozaev manifold is a natural constraint. By considering a minimizing sequence on the Pohozaev manifold, they derived existence results.
\item The Lyapunov-Schmidt reduction approach also has been applied to problem \eqref{22072316}-\eqref{2303012154}, see \cite[Section 3]{PPVV2021}.
\end{itemize}

\subsection{Existence results with Sobolev critical growth}
 It seems that all the above works can not be directly extended to the nonlinearities involving Sobolev critical growth. As far as we are aware, we did not find any works on normalized solution with Sobolev critical growth and potential $V \not\equiv 0$. Motivated by \cite{ WW22, Soave20_JFA} for the non-potential case, we consider \eqref{041003} with mass-critical or mass-supercritical, and Sobolev critical nonlinearities. For potential $V\in C^1(\mathbb{R}^N)$, we will use the following assumptions:
\begin{itemize}
\item [$(V_1)$]  $\lim\limits_{|x|\rightarrow\infty} V(x)=\underset{x\in  \mathbb{R}^N}{\sup} V(x)=:V_{\infty}<\infty$ and there exists $\sigma_1> 0$ such that
\begin{equation}\label{2301262110}
    \int_{\mathbb{R}^N}\left|V-V_\infty\right|u^2 \ud x\le \sigma_1\|\nabla u\|^2_2, \;\; \forall\, u \in H^1(\mathbb{R}^N).
\end{equation}
\item[$(V_2)$] Let $W(x):=\frac12 \langle \nabla V(x),x\rangle$, there holds $\lim\limits_{|x|\rightarrow\infty} W(x)=0$ and there is $\sigma_2> 0$ such that
\begin{equation}
    \left|\int_{\mathbb{R}^N}  \left[\frac12 (V-V_{\infty})+\frac{1}{p\gamma_p} W\right] u^2 \ud x\right|\le \sigma_2\|\nabla u\|^2_2, \;\;\forall\, u \in H^1(\mathbb{R}^N).
\end{equation}
\item[$(\widetilde{V}_2)$] There exists $\widetilde{\sigma}_2> 0$ such that
\begin{equation}\label{2301271518}
    \left|\int_{\mathbb{R}^N}Wu^2 \ud x\right|\le \widetilde{\sigma}_2\|u\|^2_{H^1}, \;\;\forall\, u \in H^1(\mathbb{R}^N).
\end{equation}
\item[$(V_3)$] $\underset{|x|\to \infty}{\lim}|x|(V-V_{\infty})(x)=0$. There exists ${\sigma_3}>0$ such that
\begin{equation}\label{2301271916}
     \int_{\mathbb{R}^N} \left(V-V_{\infty}\right)^2|x|^2u^2  \ud x\le {\sigma_3^2}\|\nabla u\|^2_2, \;\;\forall\, u \in H^1(\mathbb{R}^N).
\end{equation}
\item[$(V_4)$] $V+W\le V_{\infty}$ for any $x\in\mathbb{R}^N.$
\end{itemize}

\begin{remark}\label{2304061410}
The condition \eqref{2301271916} implies \eqref{2301262110}. In fact, using Hardy's inequality, we have
\[
\int_{\mathbb{R}^N}\left|V-V_\infty\right|u^2 \ud x\le \|(V-V_\infty)|x|u\|_2\left\|\frac{u}{|x|}\right\|_2\le \frac{2}{N-2}\sigma_3 \|\nabla u\|_2^2.
\]
Thus $\sigma_1\le \frac{2}{N-2}\sigma_3$.
\end{remark}
\begin{remark}\label{2304121417}
In $(V_1)$, if $V_\infty\neq 0$, one may replace $(V, \lambda)$ by $(V-V_\infty, \lambda+V_\infty)$. Hence, without loss of generality, we can assume that $V_\infty=0$.
\end{remark}

By virtue of Remark \ref{2304121417}, from now on, {\bf we assume that $V_\infty=0$}. Problem \eqref{041003} can be regarded as a perturbation problem of \cite{Soave20_JFA, WW22}, whose energy functional is given by
\begin{equation*}
\begin{aligned}
J_{\mu, V}(u)=\frac{1}{2}\|\nabla u\|_2^2
+\frac12 \int_{\mathbb{R}^N} V(x)u^2 \ud x-\frac{\mu}{p}\|u\|_p^{p}
-\frac{1}{2^{*}}\|u\|_{2^*}^{2^{*}}.
\end{aligned}
\end{equation*}
Recall that for $p\in [\hat{p}, 2^*)$, the solutions of  \cite[Theorem 1.1]{Soave20_JFA} and \cite[Theorem 1.1]{WW22} are mountain pass type solutions of Jeanjean. In present paper, we seek for $V$ to preserve the mountain pass geometry of $J_{\mu, V}$ on $S_a$. When different integrability assumptions are required on $V$, the Pohozaev identity for \eqref{041003} could have two different type of forms, that is,
\begin{itemize}
\item[$(a)$]  if $V\in L_{loc}^{\frac{N}{2}}(\mathbb{R}^N)$ satisfies \eqref{2301271916}, any solution to \eqref{22062610} will satisfy
\begin{equation}\label{2306261440}
P_{\mu, V}(u):=\|\nabla u\|_2^2
+\frac{N}{2}\int_{\mathbb{R}^N}V(x)u^2 \ud x +\int_{\mathbb{R}^N}V(x)u\nabla u\cdot x \ud x-\mu\gamma_p\|u\|_p^p
-\|u\|_{2^*}^{2^*}=0;
\end{equation}
\item[$(b)$] if $V\in C^1(\mathbb{R}^N)$ and \eqref{2301262110}, \eqref{2301271518} hold, any solution to \eqref{22062610} will satisfy
\begin{equation}\label{2306261441}
\widetilde{P}_{\mu, V}(u):=\|\nabla u\|_2^2
-\int_{\mathbb{R}^N} W(x)u^2 \ud x-\mu\gamma_p\|u\|_p^p
-\|u\|_{2^*}^{2^*}=0.
\end{equation}
\end{itemize}
We will work with either \eqref{2306261440} or \eqref{2306261441} following different assumptions on $V$. Our main results are the following.
\begin{theorem}\label{22071658}
Assume that $N\ge 3$, $\mu, a>0$ and $V\in C^1(\mathbb{R}^N)$ satisfies $(V_1)$, $(\widetilde{V}_2)$, $(V_3)$, $(V_4)$.
\begin{itemize}
\item[$(i)$] If $p=\hat{p}$, $\sigma_1<1$ and
\begin{equation}\label{2303021949}
\sigma_3<\frac{2}{N-2}\;\;\text{if}~N\ge 4; ~~~\text{or}~~ ~{3\sigma_1 + 2{\sigma_3}} < 4\;\; \text{if}~N=3,
\end{equation}
then there exists some $a_0=a_0(\mu, p, N)>0$ such that for any $a\in (0, a_0)$, problem \eqref{041003} has a positive solution.\par
\item[$(ii)$] If $p\in (\hat{p}, 2^*)$, $\sigma_1<1$ and
\begin{equation}
{\sigma_3} < \frac{p\gamma_p}{2} - 1\;\; \text{if}~N\ge 4;
\end{equation}
\begin{equation}
(p\gamma_p - 3) \sigma_1 + 2\sigma_3 < p\gamma_p - 2\;\; \text{if}~N=3,\, 4<p<6,
\end{equation}
then for any $\mu,\, a>0$, problem \eqref{041003} has a positive solution.
\end{itemize}
\end{theorem}
\begin{theorem}\label{2207241248}
Assume that $N\ge 3$, $\mu, a>0$ and $V\in C^1(\mathbb{R}^N)$ satisfies $(V_1)$, $(V_2)$, $(V_4)$.
\begin{itemize}
\item[$(i)$] If $p=\hat{p}$, $\sigma_1<1$ and
\begin{equation}
\sigma_2<\frac1N,
\end{equation}
there exists some $a_0=a_0(\mu, p, N)>0$ such that for any $a\in (0, a_0)$, problem \eqref{041003} has a positive solution.\par
\item[$(ii)$] If $p\in (\hat{p}, 2^*)$, $\sigma_1<1$ and
\begin{equation}\label{2303021950}
\sigma_2 < \frac12-\frac{1}{p\gamma_p},
\end{equation}
then for any $\mu,\, a>0$, problem \eqref{041003} has a positive solution.
\end{itemize}
\end{theorem}
To prove Theorems \ref{22071658}, \ref{2207241248}, we will follow Jeanjean's method \cite{Jeanjean97}. First, the condition $\sigma_1<1$
in Theorems \ref{22071658}-\ref{2207241248} is to preserve the mountain pass geometry of $J_{\mu, V}$ on $S_a$ (see Lemma \ref{2207232313}). Note that for the Sobolev critical problems, the compactness of a Palais-Smale sequence often relies on its energy level. For this reason, in section \ref{2306251339} we select the mountain pass end points $\{e_0, e_1\}$ located on the fiber of $v^a$, where $v^a$ is the solution of Theorem \ref{2207231547}. Thus, the mountain pass level of $J_{\mu,\, V}$ has an upper bound $\frac{1}{N}\mathcal{S}^{\frac{N}{2}}$ (see Lemmas \ref{2301312159}, \ref{2209120004}). 

Next, considering the augmented functional $\widetilde{J}_{\mu, V}$ given in \eqref{2307301040}, it also has a mountain pass geometry on $S_a\times \mathbb{R}^+$with the same mountain pass level as ${J}_{\mu, V}$, seeing Lemma \ref{2302152315}. Applying Proposition \ref{22072013} to $\widetilde{J}_{\mu, V}|_{S_a\times \mathbb{R}^+}$, we get a Palais-Smale sequence approaching the Pohozaev manifold
\[
\mathcal{P}_{\mu, V}:= \big\{u\in S_a: {P}_{\mu, V}(u)=0 \big\}\;\; \text{or}\;\; \widetilde{\mathcal{P}}_{\mu, V}:= \big\{u\in S_a: \widetilde{P}_{\mu, V}(u)=0 \big\}.
\]
In contrast with \cite{DZ2022}, we do not require the Pohozaev manifold as a natural constraint. Thus the Pohozaev manifold is not required to be in $C^1$, hence we just assume $V\in C^1(\mathbb{R}^N)$. Following the method of \cite{BMRV2021,MRV2022}, \eqref{2303021949}-\eqref{2303021950} are used to guarantee the boundedness of the Palais-Smale sequence, and the nonnegativity of the energy for solutions to \eqref{041003} (see Lemmas \ref{22072415}, \ref{2207241537}, \ref{22072315}). Due to the critical term and the potential term, the sign of $\lambda$ cannot simply be concluded. Different from \cite{MRV2022} and \cite{DZ2022},  we will use the assumption $(V_4)$ to compute the sign of $\lambda$ (see Lemma \ref{22062613}).

Furthermore, when considering the compactness of the Palais-Smale sequence in $H^1(\mathbb{R}^N)$, we cannot always work in the radial subspace $H_{rad}^1(\mathbb{R}^N)$ since the potential is not supposed to be radial, which brings some new difficulties. At last, the Sobolev critical growth yields further difficulties since the usual Lions' lemma (see \cite[Theorem 1.34]{Willem96}) is no longer suitable. To overcome all these obstacles, we will prove a splitting lemma (see Lemma \ref{2207110022}) inspired by \cite{BMRV2021, MRV2022} and \cite[Proposition 3.1]{Soave20_JFA}. 

When these parameters $\sigma_1$-$\sigma_3$, $\widetilde{\sigma}_2$ are 0, the equation \eqref{041003} becomes the non-potential case, and Theorems \ref{22071658}, \ref{2207241248} will come back to Theorem \ref{2207231547}.

\begin{remark}\label{2302011420}
Notice that \eqref{2301262110}-\eqref{2301271916} are a kind of compactness assumptions or Poincar\'e inequalities. Let $\alpha, \beta>0$ and
\[
V_{\alpha,\beta}=\left(\frac{\beta}{\alpha^4}|x|^2-\frac{2\beta}{\alpha^2}\right)\chi_{\{|x|\le \alpha\}} -\frac{\beta}{|x|^2} \chi_{\{|x|> \alpha\}},
\]
where $\chi$ is the characteristic function. Whenever $\alpha$, $\beta$ are positive but small enough, all the assumptions $(V_1)$--$(V_4)$ and $(\widetilde{V}_2)$ will be satisfied.
\end{remark}
\begin{remark}
Theorems \ref{22071658} and \ref{2207241248} respectively use the Pohozaev identities \eqref{2306261440} and \eqref{2306261441}. When $p=\hat{p}$, the values of $a_0$ in Theorems \ref{22071658}, \ref{2207241248} can be given from \eqref{22072321132}, \eqref{22072321133}, \eqref{22072321141} and \eqref{2301301540}. By Lemma \ref{22062613}, we also know that $\lambda<-V_\infty$ in Theorems \ref{22071658}, \ref{2207241248}.
\end{remark}

\begin{remark}\label{2302161844}
When $V\in C_{loc}^{0,\alpha}(\mathbb{R}^N, \mathbb{R})$ for some $\alpha\in (0,1)$, by standard elliptic theory (\cite[B.3 Lemma, B.2 Theorem]{Struwe2008}), any solution in $H^1(\mathbb{R}^N)$ of \eqref{041003} is a classical solution.
\end{remark}
This paper is organized as follows. In Section \ref{2306251337}, we introduce some notations and preliminary results. In Section \ref{2306251339}, we verify the mountain pass structure of $J_{\mu, V}$ on $S_a$. In section \ref{2306251338}, we derive the existence of a Palais-Smale sequence approaching Pohozaev manifold. In section \ref{2306251340}, we complete the proof of Theorem \ref{22071658} and Theorem \ref{2207241248}.

\section{Notations and preliminary results}\label{2306251337}
For $p\in [1, \infty]$, we denote by $L^p(\mathbb{R}^N)$ the Lebesgue's space with norm $\|\cdot\|_p$ and by $H^1(\mathbb{R}^N)$, $D^{1,2}(\mathbb{R}^N)$ the usual Sobolev spaces. We always denote $\gamma_p=\frac{N(p-2)}{2p}$. Let $\mathcal{S}$ be the best Sobolev constant given in \eqref{2301112226}, $C_{N, p}$ be the optimal constant for Gagliardo-Nirenberg inequality given in \eqref{22071819}. For $u\in H^1(\mathbb{R}^N)$, $u_h$ denotes the scaling transformation \eqref{2301111617}. $o_n(1)$ denotes a real sequence with $o_n(1)\to 0$ as $n\to \infty$.

For $2 \leq p \leq 2^*$, the Pohozaev identity of the equation
\begin{equation}\label{22071100}
-\Delta u=\lambda u+ \mu|u|^{p-2}u+ |u|^{2^*-2}u~ \;\;\text { in } \mathbb{R}^N,
\end{equation}
is $P_\mu(u)=0$, where for any $u\in H^1(\mathbb{R}^N)$,
\begin{equation}\label{2306281307}
P_{\mu}(u):= \|\nabla u\|_2^2-\mu \gamma_p\|u\|_p^p-\|u\|_{2^*}^{2^*}.
\end{equation}

We state a series of preliminary results. We begin with giving the Pohozaev identity for $u\in H^1(\mathbb{R}^N)$, solution to
\begin{equation}\label{22062610}
-\Delta u+V(x)u=\lambda u+\mu |u|^{p-2}u+|u|^{2^{*}-2}u \;\;\text { in } \mathbb{R}^N.
\end{equation}
Let $h > 0, u\in S_a$, we denote a fiber function of $u$ as
\begin{align}\label{22072120}
\begin{split}
  \Psi_u(h) &=J_{\mu, V}(u_h)=\frac{h^2}{2}\|\nabla u\|_2^2
+\frac{h^N}{2} \int_{\mathbb{R}^N} V(x)u^2(hx) \ud x-\frac{\mu h^{p\gamma_p }}{p}\|u\|_p^p
-\frac{h^{2^*}}{2^{*}}\|u\|_{2^*}^{2^*}\\
& =\frac{h^{2}}{2}\|\nabla u\|_2^2
+\frac12 \int_{\mathbb{R}^N} V(h^{-1}x)u^2 \ud x-\frac{\mu h^{p\gamma_p }}{p}\|u\|_p^p
-\frac{h^{2^*}}{2^{*}}\|u\|_{2^*}^{2^*}.
\end{split}
\end{align}
Using Lemma \ref{2301272023} (see Appendix), the condition \eqref{2301271916} yields $\Psi'_u(1)={P}_{\mu, V}(u)$ where ${P}_{\mu, V}$ is given in \eqref{2306261440}. On the other hand, if $V\in C^1(\mathbb{R}^N)$ and \eqref{2301262110}, \eqref{2301271518} hold, then
$\Psi'_u(1)=\widetilde{P}_{\mu, V}(u)$ where $\widetilde{P}_{\mu, V}$ is defined in \eqref{2306261441}.
Using Proposition \ref{2302171335}, any solution of \eqref{22062610} satisfies the Pohozaev identity
\begin{equation}\label{2302071001}
\begin{cases}
{P}_{\mu, V}(u)=0\;\; \text{if ~} V\in L_{loc}^{\frac{N}{2}}(\mathbb{R}^N)~\text{and}~\eqref{2301271916}~\text{hold},\\
\widetilde{P}_{\mu, V}(u)=0\;\; \text{if ~} V\in C^1(\mathbb{R}^N) ~\text{and}~\eqref{2301262110},\,\eqref{2301271518}~\text{hold}.
\end{cases}
\end{equation}\par
The following result comes from \cite[Lemma 3.3]{WW22}, it showed the monotonicity and the bound of $m_a$.
\begin{lemma}\label{2301312159}
Let $N\ge 3$ and $\hat{p}\le p<2^*$. Then there exists $\alpha_{N, p}$ only depending on $N, p$, such that $m_{a}$ is strictly decreasing for $0<\mu<a^{{p\gamma_p-p}}\alpha_{N, p}$ and is nonincreasing for $\mu\ge a^{{p\gamma_p-p}}\alpha_{N, p}$, where $m_a$ is given by \eqref{22071824}. Moreover, $0<m_a<\frac{1}{N}\mathcal{S}^{\frac{N}{2}}$ for all $\mu > 0$ if $\hat{p} < p<2^*$, while $m_a=0$ for $\mu\ge a^{{p\gamma_p-p}}\alpha_{N, p}$ if $p=\hat{p}$.
\end{lemma}
As already mentioned, when $p=\hat{p}$, there holds $\alpha_{N, \hat{p}}=\frac{\hat{p}}{2(C_{N, \hat{p}})^{\hat{p}}}$, see the formula\cite[(3.13)]{WW22}. Note that when $\mu, a$ satisfy
\begin{equation}\label{2301291935}
\mu,~ a>0\;\;\text{if}~p\in (\hat{p}, 2^*); \;\; \mbox{or }\;\;
0<\mu a^{\frac4N}<\frac{\hat{p}}{2(C_{N, \hat{p}})^{\hat{p}}}\;\; \text{if}~p=\hat{p},
\end{equation}
the functional $J_{\mu}$ has a mountain pass geometry on $S_a$. In particular, for any $u\in S_a$, $J_{\mu}(u_h)$ has a unique maximal point on $\mathbb{R}^+$ and
\begin{equation}\label{2302152221}
m_{a}=\underset{u\in S_a}{\inf}\underset{h\in \mathbb{R}^+}{\max}J_{\mu}(u_h).
\end{equation}
Here $u_h$ is given in \eqref{2301111617}. We have the following monotonicity result.
\begin{lemma}\label{2207092032}
Assume that $N\ge 3$, $\hat{p}\le p<2^*$, $\mu, \, a$ satisfy \eqref{2301291935} and $m_a$ is defined in \eqref{22071824}. Then $m_a$ is non-increasing with respect to $a$.
\end{lemma}
\begin{proof}
Let $a_2>a_1>0$, and $\theta=\frac{a_2}{a_1}>1$. For any $u_1\in S_{a_1}$, let
\[
u_2(x):=\theta^{\frac{2-N}{2}}u_1(\theta^{-1}x).
\]
By a direct computation, we have $\|\nabla u_2\|_2=\|\nabla u_1\|_2$, $\|u_2\|_2 =\theta \|u_1\|_2$, $\|u_2\|_p^p=\theta^{\frac{p(2-N)+2N}{2}} \|u_1\|_p^p$ and $\|u_2\|_{2^*} =\|u_1\|_{2^*}$. By \eqref{2302152221}, there exists some $h_{u_2}$ such that
\[
\begin{aligned}
m_{a_2} &\le \underset{h\in \mathbb{R}^+}{\max} J_{\mu}((u_2)_{h})=J_{\mu}((u_2)_{h_{u_2}})\\
&=\frac12 \|\nabla (u_1)_{h_{u_2}}\|_2^2-\frac{\mu}{p}\theta^{\frac{p(2-N)+2N}{2}}\|(u_1)_{h_{u_2}}\|_p^p-\frac{1}{2^*}\|(u_1)_{h_{u_2}}\|_{2^*}^{2^*}\\
&\le \frac12 \|\nabla (u_1)_{h_{u_2}}\|_2^2-\frac{\mu}{p}\|(u_1)_{h_{u_2}}\|_p^p-\frac{1}{2^*}\|(u_1)_{h_{u_2}}\|_{2^*}^{2^*}\\
&\le \underset{h\in \mathbb{R}^+}{\max} J_{\mu}((u_1)_{h}).\\
\end{aligned}
\]
By the definition of $m_{a_1}$, we have  $m_{a_1}\ge m_{a_2}$.
\end{proof}
In the sequel, let $P_\mu$ be given in \eqref{2306281307}.
We have the following lemma.
\begin{lemma}\label{22070919}
Assume that $N\ge 3$, $\hat{p}\le p<2^*$ and $\mu,\, a$ satisfy \eqref{2301291935}. Let $\{u_n\}\subset S_{a}$ be a bounded sequence in $H^1(\mathbb{R}^N)$ such that $P_{\mu}(u_n)\rightarrow 0$. If $\|\nabla u_n\|_2^2>c > 0$, then
\begin{equation}\label{2207272338}
m_a\le \underset{n\rightarrow \infty}{\lim\inf}J_{\mu}(u_n).
\end{equation}
\end{lemma}
\begin{proof}
Since $P_{\mu}(u_n)\rightarrow 0$, we have \begin{equation}\label{22070917}
\|\nabla u_n\|_2^2-\mu\gamma_p\|u_n\|_p^p-\|u_n\|_{2^*}^{2^*}=o_n(1).
\end{equation}
There exists a sequence $h_n>0$ such that $(u_n)_{h_n}\in \mathcal{P}_{\mu}$, that is,
\[
\|\nabla (u_n)_{h_n}\|_2^2-\mu\gamma_p\|(u_n)_{h_n}\|_p^p-\|(u_n)_{h_n}\|_{2^*}^{2^*}=0,
\]
where $(u_n)_{h_n}$ is given in \eqref{2301111617}. By direct computation,
\begin{equation}\label{2207091725}
\begin{aligned}
h_n^2\|\nabla u_n\|_2^2-\mu\gamma_p  h_n^{p\gamma_p}\|u_n\|_p^p-h_n^{2^*}\|u_n\|_{2^*}^{2^*}=0.\\
\end{aligned}
\end{equation}
Combining \eqref{22070917} with \eqref{2207091725}, we obtain
\[
\mu\gamma_p ( h_n^{p\gamma_p-2}-1)\|u_n\|_p^p+(h_n^{2^*-2}-1)\|u_n\|_{2^*}^{2^*}=o_n(1).
\]
From $\|\nabla u_n\|_2^2>c > 0$ and \eqref{22070917}, it follows that $\{\|u_n\|_{2^*}^{2^*}\}$ has a positive lower bound, which deduces that $h_n\rightarrow 1$. So
\[
J_{\mu}(u_n)=J_{\mu}((u_n)_{h_n})+o_n(1)\ge m_a + o_n(1).
\]
Thus \eqref{2207272338} holds.
\end{proof}
Now, we claim a non-existence result of solutions to \eqref{22062610}, which yields the sign of the Lagrange multiplier $\lambda$ in \eqref{041003}.
\begin{lemma}\label{22062613}
Suppose that $N\ge 3$, $p\in (2, 2^*)$, and $V\in C^1(\mathbb{R}^N)$ satisfies \eqref{2301262110}, $(\widetilde{V}_2)$, $(V_4)$, then the equation \eqref{22062610}
has no nontrivial solution provided $\lambda\ge 0$.
\end{lemma}
\begin{proof}
Assume that $u$ is a nontrivial solution to \eqref{22062610}. Multiplying \eqref{22062610} by $u$, we obtain
\begin{equation}\label{22061001}
\|\nabla u\|_2^2+\int_{\mathbb{R}^N} V(x)u^2 \ud x=\lambda\|u\|_2^2+\mu \|u\|_p^p+\|u\|_{2^*}^{2^*}.
\end{equation}
By Proposition \ref{2302171335}, the Pohozaev identity \eqref{2306261441} holds. Combining \eqref{22061001} with \eqref{2306261441}, we have
\begin{equation}
\int_{\mathbb{R}^N} (V(x)+W(x))u^2 \ud x=\mu (1-\gamma_p) \|u\|_p^p+\lambda\|u\|^2_2.
\end{equation}
From $(V_4)$ and $\gamma_p<1$, it follows that $\lambda<0$.
\end{proof}
Next, we are going to prove that under suitable condition, the functional $J_{\mu,\, V}$ for any solution to \eqref{22062610} is nonnegative.
\begin{lemma}\label{22072415}
Suppose that $N\ge 3$, $p\in [\hat{p}, 2^*)$, and $V\in L_{loc}^{\frac{N}{2}}(\mathbb{R}^N)$ satisfies $(V_1)$ and $(V_3)$ with
\begin{itemize}
\item[$(i)$] \text{either} $p=\hat{p}$,
\begin{equation}\label{22072321132}
{\sigma_3} <  \frac2{N-2}\left(1-\frac{2\mu}{\hat{p}}(C_{N, \hat{p}})^{\hat{p}}a^{\frac4N}\right)\;\; \text{if}~N\ge 4;
\end{equation}
\begin{equation}\label{22072321133}
 3\sigma_1 + 2\sigma_3 < 4\left(1-\frac{2\mu}{\hat{p}}(C_{3,\hat{p}})^{\hat{p}} a^{\frac43}\right)\;\; \text{if}~N=3;
\end{equation}
\item[$(ii)$] \text{or} $p\in (\hat{p}, 2^*)$,
\begin{equation}\label{2207232113}
{\sigma_3}< \frac{p\gamma_p}{2}-1\;\; \text{if}~N\ge 4;
\end{equation}
\end{itemize}
\begin{equation}\label{22072321131}
(p\gamma_p-3) \sigma_1 + 2 \sigma_3 < p\gamma_p - 2\;\; \text{if}~N=3,\, 4<p<6.
\end{equation}
Then any solution to \eqref{22062610} satisfies $J_{\mu, V}(u)\ge 0$.
\end{lemma}
\begin{proof}
Since $u$ is a solution to \eqref{22062610}, by Proposition \ref{2302171335}, $u$ satisfies the Pohozaev identity \eqref{2306261440}.
First, suppose $p\in (2+\frac{4}{N}, 2^*)$. Taking \eqref{2306261440} into $J_{\mu, V}$, we get
\[
\begin{aligned}
J_{\mu, V}(u)&=\frac12 \|\nabla u\|_2^2+\frac12 \int_{\mathbb{R}^N} V(x)u^2 \ud x-\frac{\mu}{p} \|u\|_p^p-\frac{1}{2^*}\|u\|_{2^*}^{2^*}\\
&\ge \frac12 \|\nabla u\|_2^2+\frac12 \int_{\mathbb{R}^N} V(x)u^2 \ud x-\frac{1}{p\gamma_p}\left(\mu \gamma_p \|u\|_p^p+\|u\|_{2^*}^{2^*}\right)\\
&= \left(\frac12-\frac{1}{p\gamma_p}\right) \|\nabla u\|_2^2+\left(\frac12-\frac{N}{2p\gamma_p}\right) \int_{\mathbb{R}^N} V(x)u^2 \ud x-\frac{1}{p\gamma_p}\int_{\mathbb{R}^N} V(x)u\nabla u\cdot x \ud x \\
\end{aligned}
\]
If $N\ge 4$ (hence $N\ge p\gamma_p$ for all $p \geq \hat p$), by H\"{o}lder inequality, $(V_3)$, $V(x)\le 0$ and \eqref{2207232113}, we have
\[
\begin{aligned}
J_{\mu, V}(u)\ge \left(\frac12-\frac{1}{p\gamma_p}\right) \|\nabla u\|_2^2-\frac{1}{p\gamma_p} {\sigma_3}\|\nabla u\|^2_2 \ge 0.
\end{aligned}
\]
If $N=3$, $4<p<6$,  we get $N< p\gamma_p$. Using H\"{o}lder inequality, $(V_1)$, $(V_3)$ and \eqref{22072321131}, there holds
\[
\begin{aligned}
J_{\mu, V}(u)\ge \left(\frac12-\frac{1}{p\gamma_p}\right) \|\nabla u\|_2^2-\left(\frac12-\frac{3}{2p\gamma_p}\right) \sigma_1\|\nabla u\|_{2}^2-\frac{1}{p\gamma_p} {\sigma_3}\|\nabla u\|^2_2 \ge 0.
\end{aligned}
\]
Suppose now $p=\hat{p}$. Similarly as above, taking \eqref{2306261440} into $J_{\mu, V}$ and eliminating the term $\|u\|_{2^*}^{2^*}$, we get
\[
\begin{aligned}
J_{\mu, V}(u)= &\left( \frac12- \frac1{2^*}\right)\|\nabla u\|_2^2+\frac12\left(1 - \frac{N}{2^*}\right)\int_{\mathbb{R}^N} V(x)u^2 \ud x-\left(\frac{1}{{\hat{p}}} -\frac{\gamma_{\hat{p}}}{2^*}\right)\mu\|u\|_{\hat{p}}^{\hat{p}}\\
&-\frac1{2^*}\int_{\mathbb{R}^N} V(x)u\nabla u\cdot x \ud x \\
\ge& \left[\frac12- \frac1{2^*}-\left(\frac{1}{{\hat{p}}} -\frac{\gamma_{\hat{p}}}{2^*}\right)\mu (C_{N,\, {\hat{p}}})^{{\hat{p}}}a^{\frac4N}-\frac{{\sigma_3}}{2^*}\right]\|\nabla u\|_2^2+\frac12\left(1 - \frac{N}{2^*}\right)\int_{\mathbb{R}^N} V(x)u^2 \ud x
\end{aligned}
\]
If $N\ge 4$ (so $N\ge 2^*$), by $V(x)\le 0$ and \eqref{22072321132}, we get
\[
J_{\mu, V}(u)\ge \frac1{2^*}\left[ \frac2{N-2}\left(1-\mu \gamma_{\hat{p}}(C_{N, {\hat{p}}})^{{\hat{p}}}a^{\frac4N}\right)- \sigma_3\right]\|\nabla u\|_2^2\ge 0.
\]
If $N=3$, then $N<2^*$. By $(V_1)$ and \eqref{22072321133}, there holds
\[
J_{\mu, V}(u)\ge \frac12\left[ \frac2{3}\left(1-\mu \gamma_{\hat{p}}(C_{3, \hat{p}})^{\hat{p}} a^{\frac43}\right) -\frac{1}{3}{\sigma_3}- \frac{1}{2}\sigma_1\right]\|\nabla u\|_2^2\ge 0.
\]
The proof is completed.
\end{proof}
\begin{lemma}\label{2207241537}
Suppose that $N\ge 3$, $p\in [\hat{p}, 2^*)$ and $V\in C^1(\mathbb{R}^N)$ satisfies $(V_1)$, $(V_2)$ with
\begin{itemize}
\item[$(i)$] either $p=\hat{p}$,
\begin{equation}\label{22072321141}
\sigma_2 < \frac{1}{N}\left(1-\frac{2\mu}{\hat{p}}(C_{N, \hat{p}})^{\hat{p}}a^{\frac4N}\right);
\end{equation}
\item[$(ii)$] or $p\in (\hat{p}, 2^*)$,
\begin{equation}\label{2207232114}
\sigma_2 < \frac12-\frac{1}{p\gamma_p}.
\end{equation}
\end{itemize}
If $u$ is a solution to \eqref{22062610}, then $J_{\mu, V}(u)\ge 0$.
\end{lemma}
\begin{proof}
Since $u$ is a solution to \eqref{22062610}, by Proposition \ref{2302171335}, $u$ satisfies the Pohozaev identity \eqref{2306261441}.
Suppose $p\in (\hat{p}, 2^*)$. Taking \eqref{2306261441} into $J_{\mu, V}$ and using $(V_1)$, $(V_2)$ and \eqref{2207232114}, we have
\[
\begin{aligned}
J_{\mu, V}(u)&\ge \frac{1}{2}\|\nabla u\|_2^2
+\frac12 \int_{\mathbb{R}^N} V(x)u^2 \ud x-\frac{1}{p\gamma_p}\left(\mu \gamma_p \|u\|_p^p+\|u\|_{2^*}^{2^*}\right)\\
&= \frac{1}{2}\|\nabla u\|_2^2
+\frac12 \int_{\mathbb{R}^N} V(x)u^2 \ud x-\frac{1}{p\gamma_p}\left( \|\nabla u\|_2^2-\int_{\mathbb{R}^N} W(x)u^2 \ud x \right)\\
&= \left(\frac{1}{2}-\frac{1}{p\gamma_p}\right)\|\nabla u\|_2^2
+\frac12 \int_{\mathbb{R}^N} V(x)u^2 \ud x+\frac{1}{p\gamma_p}\int_{\mathbb{R}^N} W(x)u^2 \ud x\\
&\ge \left(\frac{1}{2}-\frac{1}{p\gamma_p}-\sigma_2\right)\|\nabla u\|_2^2\\
&\ge 0.
\end{aligned}
\]
If now $p=\hat{p}$. Taking \eqref{2306261441} into $J_{\mu, V}$ and eliminating the term $\|u\|_{2^*}^{2^*}$ and by $(V_1)$, $(V_2)$, \eqref{22072321141}, there holds
\[
\begin{aligned}
J_{\mu, V}(u)&= \left(\frac{1}{2}-\frac1{2^*}\right)\|\nabla u\|_2^2
+\frac12 \int_{\mathbb{R}^N} V(x)u^2 \ud x+\frac1{2^*} \int_{\mathbb{R}^N} W(x)u^2 \ud x-\left(\frac{1}{p} -\frac{\gamma_{\hat{p}}}{2^*}\right)\mu\|u\|_{\hat{p}}^{\hat{p}}
\\
&\ge \left( \frac12\left(1-\frac2{2^*}\right)\left(1-\mu \gamma_{\hat{p}}C_{N,\, {\hat{p}}}^{{\hat{p}}}a^{\frac4N}\right)-\sigma_2\right)\|\nabla u\|_2^2\\
&\ge 0.
\end{aligned}
\]
The proof is completed.
\end{proof}
To find a Palais-Smale sequence approaching Pohozaev manifold, the following frame will be often used, which is a special case of \cite[Theorem 4.5]{Ghoussoub1993} (see also \cite[Lemma 3.1]{BMRV2021}):
\begin{proposition}\label{22072013}
Let $X$ be a Hilbert manifold and $J\in C^1(X,\mathbb{R})$. Let $K\subset X$ be compact and consider a family
\[
\mathcal{E}\subset \{E\subset X: E ~\text{is}~ \text{compact},~K\subset E \}
\]
which is invariant with respect to all deformations leaving $K$ fixed. Assume that
\[
\underset{u\in K}{\max}J(u)<c:=\underset{E\in \mathcal{E}}{\inf}\underset{u\in E}{\max}J(u)\in \mathbb{R}.
\]
Let $\sigma_n\in \mathbb{R}$, $\sigma_n\to 0$ and $E_n\in \mathcal{E}$ be a sequence such that
\[
c\le \underset{u\in E_n}{\max} J(u)<c+\sigma_n.
\]
Then there exists a sequence $v_n\in X$ such that for some $C'>0$,
$$ (i)\;  c\le J(v_n) <c+\sigma_n; \;\; (ii)\;  \|\nabla_{X}J(v_n)\|< C' \sqrt{\sigma_n}; \;\; (iii) \; \text{dist}(v_n, E_n)< C'\sqrt{\sigma_n}.$$
\end{proposition}

\section{Mountain pass structure}\label{2306251339}
In this section, we are devoted to construct the mountain pass geometry of $J_{\mu, V}$ on $S_a$. Let $u_h$ be given in \eqref{2301111617}.
\begin{lemma}\label{2207232312}
Assume that $N\ge 3$, $p\in (2, 2^*)$, and $(V_1)$ holds. Then for any $u\in S_a$, $\underset{h\to 0^+}{\lim}J_{\mu, V}(u_h)=0$ and  $\underset{h\to \infty}{\lim}J_{\mu, V}(u_h)=-\infty$.
\end{lemma}
\begin{proof}
Let $u\in S_a$. By direct computation, we have
\[
\begin{aligned}
J_{\mu, V}(u_h)\le \frac{h^{2}}{2} \|\nabla u\|_2^2-\frac{\mu h^{p\gamma_p }}{p} \|u\|_p^p-\frac{h^{2^*}}{2^*}\|u\|_{2^*}^{2^*}.\\
\end{aligned}
\]
For $p\in (2, 2^*)$, we deduce that $\underset{h\to \infty}{\lim}J_{\mu, V}(u_h)=-\infty$.
On the other hand, by $(V_1)$,
\[
\begin{aligned}
\left| J_{\mu, V}(u_h)\right|\le \left(\frac12+\sigma_1\right) h^{2}\|\nabla u\|_2^2+\frac{\mu h^{p\gamma_p }}{p} \|u\|_p^p+\frac{h^{2^*}}{2^*}\|u\|_{2^*}^{2^*}.\\
\end{aligned}
\]
Thus $\underset{h\to 0^+}{\lim}J_{\mu, V}(u_h)=0$.
\end{proof}
Now, we consider for $k>0$,
\[
A_k=\{u\in S_a: \|\nabla u\|_2< k\}\;\;\text{and} \;\; \partial A_k=\{u\in S_a: \|\nabla u\|_2= k\}.
\]
A key observation is that under $(V_1)$, $J_{\mu, V}$ is lower bounded by a positive constant on $\partial A_k$ for $k>0$ small enough.
\begin{lemma}\label{2207232313}
Assume that $N\ge 3$, $p\in [\hat{p}, 2^*)$, and $(V_1)$ holds with
\begin{equation}\label{2301301540}
\sigma_1<1\;\; \text{if}~p\in (\hat{p}, 2^*);\;\; \text{or}\;\;\sigma_1<{1}-\frac{2\mu}{\hat{p}}(C_{N, \hat{p}})^{\hat{p}}a^{\frac4N}\;\;\, \text{if}~p=\hat{p}.\\
\end{equation}
Then there is some $k_0>0$ such that $J_{\mu, V}$ has a positive lower bound $\beta$ on $\partial A_{k_0}$.
\end{lemma}
\begin{proof}
By Gagliardo-Nirenberg inequality and Sobolev inequality,
\[
\begin{aligned}
J_{\mu, V}(u)&\ge \frac12 \|\nabla u\|_2^2-\frac12 \sigma_1\|\nabla u\|_{2}^2-\frac{\mu}{p} \|u\|_p^p-\frac{1}{2^*}\|u\|_{2^*}^{2^*}\\
&\ge \frac{1-\sigma_1}{2}\|\nabla u\|_2^2-\frac{\mu (C_{N,p})^p}{p}a^{{p(1-\gamma_p)}}\|\nabla u\|_2^{p\gamma_p}-\frac{1}{2^*}\mathcal{S}^{-\frac{2^*}{2}}\|\nabla u\|_2^{2^*}.
\end{aligned}
\]
Then there exist $k_0>0$ and  $\beta>0$ such that $J_{\mu, V}(u)>\beta$ for $u\in \partial A_{k_0}$.
\end{proof}
Note that by Lemmas \ref{2207232312} and \ref{2207232313}, the condition \eqref{2301301540} can preserve the mountain pass geometry of $J_{\mu, V}$. Moreover, \eqref{2301301540} implies \eqref{2301291935}, which yields that $m_a$ is well defined. Let $v^a$ be the solution of Theorem \ref{2207231547}. That is,
\begin{equation}\label{2303272013}
J_{\mu}(v^a)=m_a.
\end{equation}
According to Lemma \ref{2207232312}, we have that $\underset{h\to 0^+}{\lim}J_{\mu, V}(v^a_h)=0$ and  $\underset{h\to \infty}{\lim}J_{\mu, V}(v^a_h)=-\infty$. Hence there exist $h_1>h_0>0$ such that $e_0:=v^a_{h_0}\in A_{k_0}$, $e_1:=v^a_{h_1}\in S_a\backslash A_{k_0}$,
$J_{\mu, V}(e_0)<\beta$ and $J_{\mu, V}(e_1)<0$, where $\beta$ is given in Lemma \ref{2207232313}. Thus, the mountain pass paths are given by
\[
\Gamma:=\{\xi\in C([0,1], S_a): \xi(0)=e_0,\, \xi(1)=e_1\},
\]
and the mountain pass level is
\[
m_{V,\, a}:=\underset{\xi\in\Gamma}{\inf} \underset{t\in [0,1] }{\max} {J_{\mu, V}(\xi(t))}.
\]
Obviously,
\begin{equation}\label{2209120004}
m_{V,\, a}\ge\beta>0.
\end{equation}
Moreover, $m_{V,\, a}$ has an upper bound $m_a$ given in \eqref{22071824}.
\begin{lemma}\label{2207091943}
Assume that $N\ge 3$, $p\in [\hat{p}, 2^*)$, $(V_1)$ holds with \eqref{2301301540} and $V\not\equiv 0$. Then $m_a>m_{V,\, a}$.
\end{lemma}
\begin{proof}
Let $v^a$ be given in \eqref{2303272013}. By Theorem \ref{2207231547}, $v^a$ is positive. By Lemma \ref{2207232312}, there exists some $\widetilde{h}$ such that $\underset{h\in \mathbb{R}^+}{\max} J_{\mu, V}(v^a_h)= J_{\mu, V}( v^a_{\widetilde{h}})$. Therefore, using \eqref{2302152221}, we have
\[
\begin{aligned}
m_{V,\, a} &\le \underset{h\in \mathbb{R}^+}{\max} J_{\mu, V}(v^a_h)= J_{\mu, V}(v^a_{\widetilde{h}})< J_{\mu}( v^a_{\widetilde{h}}) \le \underset{h\in \mathbb{R}^+}{\max} J_{\mu}(v^a_{h})=m_a.
\end{aligned}
\]
So we are done.
\end{proof}
In order to derive a bounded Palais-Smale sequence of $J_{\mu, V}$ at $m_{V,\, a}$, we follow the idea in \cite{Jeanjean97} (see also \cite{BMRV2021,MRV2022}) to construct an augmented functional
\begin{equation}\label{2307301040}
\widetilde{J}_{\mu, V}(u, h):=J_{\mu, V}(u_h),\;\; (u, h)\in H^1(\mathbb{R}^N)\times \mathbb{R}^+
\end{equation}
where $u_h$ is defined in \eqref{2301111617}. Actually $\widetilde{J}_{\mu, V}$ also possesses the mountain pass structure on $S_a\times \mathbb{R}^+$. The new mountain pass paths on $S_a\times \mathbb{R}^+$ is given by
\[
\widetilde{\Gamma}:=\{\widetilde{\xi}\in C([0,1], S_a\times \mathbb{R}^+): \widetilde{\xi}(0)=(e_0, 1),\, \widetilde{\xi}(1)=(e_1,1)\}.
\]
Any path $\widetilde{\xi}\in\widetilde{\Gamma}$ can be represented by
\[
\widetilde{\xi}(t)=(\xi(t), s(t))\;\; \text{with}~\xi\in \Gamma,~s\in C([0,1], \mathbb{R}^+)~\text{and}~s(0)=s(1)=1,
\]
which will intersect the set
\[
\widetilde{\partial A_{k_0}}:=\{(u, h)\in S_a\times \mathbb{R}^+: u_h\in \partial A_{k_0}\}.
\]
Moreover, $\widetilde{J}_{\mu, V}(u)\ge \beta$ for any $u\in \widetilde{\partial A_{k_0}}$. The new mountain pass level is
\[
\widetilde{m}_{V,\, a} :=\underset{\widetilde{\xi}\in\widetilde{\Gamma}}{\inf} \underset{t\in [0,1] }{\max} {\widetilde{J}_{\mu, V}(\widetilde{\xi}(t))}.
\]
\begin{lemma}\label{2302152315}
Assume that $N\ge 3$, $p\in [\hat{p}, 2^*)$ and $(V_1)$ holds with \eqref{2301301540}. Then $\widetilde{m}_{V,\, a}={m}_{V,\, a}$.
\end{lemma}
\begin{proof}
Clearly, $\widetilde{m}_{V,\, a}$ and ${m}_{V,\, a}$ are well defined. For any $\xi\in\Gamma$, we have $(\xi, 1)\in\widetilde{\Gamma}$. Thus
\[
\widetilde{m}_{V,\, a} \le \underset{\xi\in\Gamma}{\inf} \underset{t\in [0,1] }{\max} {\widetilde{J}_{\mu, V}(\xi(t),1)}=\underset{\xi\in\Gamma}{\inf} \underset{t\in [0,1] }{\max} {J_{\mu, V}(\xi(t))}={m}_{V,\, a} .
\]
On the other hand, for any $\widetilde{\xi}\in\widetilde{\Gamma}$, it can be represented by $\widetilde{\xi}=(\xi,s)$ with $\xi\in \Gamma,~s\in C([0,1], \mathbb{R}^+)~\text{and}~s(0)=s(1)=1$. One can get $\xi_s\in \Gamma$ where $\xi_s$ is given by \eqref{2301111617}. Hence
\[
{m}_{V,\, a} \le \underset{t\in[0,1]}{\max}J_{\mu, V}(\xi_s(t))=\underset{t\in[0,1]}{\max}\widetilde{J}_{\mu, V}(\xi(t),s(t)).
\]
Since $\widetilde{\xi}$ is arbitrary, we conclude ${m}_{V,\, a} \le \widetilde{m}_{V,\, a} $. So $\widetilde{m}_{V,\, a}={m}_{V,\, a}$.
\end{proof}

\section{Palais-Smale sequence}\label{2306251338}
In this section, we are dedicated to find a bounded Palais-Smale sequence $\{u_n\}$ approaching the Pohozaev manifold, that is, $P_\mu(u_n)=o_n(1)$ or $\widetilde{P}_\mu(u_n)=o_n(1)$. Recall that a Palais-Smale sequence $\{u_n\}$ of $J_{\mu,\, V}|_{S_a}$ at $c$, means that $u_n\in S_a$,
\[
J_{\mu,\, V}(u_n)=c+o_n(1),
\]
and
\[
\|J'_{\mu,\, V}(u_n)\|_{T^*_{u_n}S_a}=o_n(1),
\]
where
\[
T_{u_n}S_a=\left\{v\in H^1(\mathbb{R}^N): \int_{\mathbb{R}^N}u_n v \ud x=0\right\}
\]
is a subspace of $H^1(\mathbb{R}^N)$, and $T^{*}_{u_n}S_a$ is the dual space of $T_{u_n}S_a$. By means of Proposition \ref{22072013}, we will obtain the following Lemma.
\begin{lemma}\label{22072020}
Assume that $N\ge 3$, $p\in [\hat{p}, 2^*)$. We have the following assertions:
\begin{itemize}
\item[$(i)$]  If $(V_1)$, $(V_3)$ hold with \eqref{2301301540}, then there exists a Palais-Smale sequence $\{u_n\}$ of $J_{\mu, V}|_{S_a}$ at the mountain pass level ${m}_{V,\, a}$, which satisfies
\begin{equation}\label{22072019}
\begin{aligned}
&P_{\mu, V}(u_n)\to 0\;\;\text{and}\;\;\|u_n^{-}\|_2\to 0\;\; \text{as}~ n\to \infty.
\end{aligned}
\end{equation}
\item[$(ii)$] If $V\in C^1(\mathbb{R}^N)$ satisfies $(V_1)$, $(V_2)$ with \eqref{2301301540}, then there exists a Palais-Smale sequence $\{u_n\}$ of $J_{\mu, V}|_{S_a}$ at the mountain pass level ${m}_{V,\, a}$, which satisfies
\begin{equation}\label{2209112237}
\begin{aligned}
\widetilde{P}_{\mu, V}(u_n)\to 0\;\;\text{and}\;\;\|u_n^{-}\|_2\to 0\;\; \text{as} ~ n\to \infty.
\end{aligned}
\end{equation}
\end{itemize}
\end{lemma}
\begin{proof}
According to the definition of ${m}_{V,\, a}$ and Lemma \ref{2302152315}, there is a minimizing sequence $\{\xi_n\}\subset \Gamma$ for ${m}_{V,\, a}$, which satisfies
\[
\widetilde{m}_{V,\, a}={m}_{V,\, a}\le \underset{t\in [0,1]}{\max}J_{\mu, V}(\xi_n(t))< {m}_{V,\, a}+\frac1n=\widetilde{m}_{V,\, a}+\frac1n.
\]
Here the minimizing sequence $\{\xi_n\}$ can be replaced by $\{|\xi_n|\}$ since $J_{\mu, V}(u)\ge J_{\mu, V}(|u|)$ for all $u\in S_a$. Therefore we can assume that $\xi_n\ge 0$.
We will prove this lemma by applying Proposition \ref{22072013} to $\widetilde{J}_{\mu, V}$ with
\[
X=S_a\times \mathbb{R}^+,\;\; K=\{(e_0, 1), (e_1, 1)\},\;\; \mathcal{E}=\{{\rm Im}({\widetilde \xi}): {\widetilde \xi}\in\widetilde{\Gamma}\},\;\; E_n=\{(\xi_n(t), 1): t\in [0,1]\},
\]
and $\sigma_n=\frac1n$. It results that there exist a sequence $\{(v_n,h_n)\}\subset S_a\times \mathbb{R}^+$ and $c>0$ such that
\begin{equation}\label{2302152329}
\begin{cases}
\begin{aligned}
\widetilde{m}_{V,\, a}\le \widetilde{J}_{\mu, V}(v_n,h_n) <\widetilde{m}_{V,\, a}+\frac1n,\\
 \|\nabla_{S_a\times \mathbb{R}^+}\widetilde{J}_{\mu, V}(v_n,h_n)\|_{T^{*}_{(v_n,h_n)}(S_a\times \mathbb{R}^+)}<c\frac{1}{\sqrt{n}},\\
\underset{t\in [0,1]}{\min}\|\xi_n(t)-v_n\|_{H^1}+|h_n-1|<c\frac{1}{\sqrt{n}},\\
\end{aligned}
\end{cases}
\end{equation}
where
$$T_{(v_n,h_n)}(S_a\times \mathbb{R}^+)=\left\{(u,t):u\in H^1(\mathbb{R}^N),\, t\in \mathbb{R},\,  \int_{\mathbb{R}^N}uv_n \ud x=0\right\}$$
which is a subspace of $H^1(\mathbb{R}^N)\times \mathbb{R}$, and $T^{*}_{(v_n,h_n)}(S_a\times \mathbb{R}^+)$ is the dual space of $T_{(v_n,h_n)}(S_a\times \mathbb{R}^+)$. In virtue of Lemma \ref{2301272023}, if $(V_1)$, $(V_3)$ hold, by differentiating $\widetilde{J}_{\mu, V}$ with respect to second variable, we get from \eqref{2302152329} that
\begin{equation}\label{2307022302}
\left|\frac{\partial \widetilde{J}_{\mu, V}}{\partial h}(v_n,h_n)\right|=|{P}_{\mu, V}((v_n)_{h_n})|\to 0;
\end{equation}
if $V\in C^1(\mathbb{R}^N)$ satisfies $(V_1)$, $(V_2)$, then
\begin{equation}\label{2307022303}
\left|\frac{\partial \widetilde{J}_{\mu, V}}{\partial h}(v_n,h_n)\right|=|\widetilde{P}_{\mu, V}((v_n)_{h_n})|\to 0,
\end{equation}
as $n\to \infty$. We differentiate $\widetilde{J}_{\mu, V}$ with respect to the first variable. It follows that
\[
\|J'_{\mu,\, V}((v_n)_{h_n})\|_{T^*_{(v_n)_{h_n}}S_a}\to 0,
\]
as $n\to \infty$, see \cite{Jeanjean97}. Setting $u_n:=(v_n)_{h_n}$ and by $h_n\to 1$, we conclude from \eqref{2302152329}, \eqref{2307022302}, \eqref{2307022303} that \eqref{22072019}, \eqref{2209112237} hold, and $\{u_n\}$ is a Palais-Smale sequence of $J_{\mu, V}$ at ${m}_{V,\, a}$.
\end{proof}
Under the assumptions $(V_1)$, $(V_2)$ or $(V_3)$, we are ready to prove the existence of bounded Palais-Smale sequence.
\begin{lemma}\label{22072315}
Let $N\ge 3$, $p\in [\hat{p}, 2^*)$. Suppose that $V\in C^1(\mathbb{R}^N)$ satisfies
\begin{itemize}
\item[$(i)$] either $(V_1)$, $(V_3)$ hold with \eqref{2301301540}, \eqref{22072321132}-\eqref{22072321131};
\item[$(ii)$] or $(V_1)$, $(V_2)$ hold with \eqref{2301301540}, \eqref{22072321141}, \eqref{2207232114}.
\end{itemize}
Let $\{u_n\}$ be a Palais-Smale sequence of $J_{\mu, V}$ at ${m}_{V,\, a}$ such that $P_{\mu, V}(u_n)\to 0$ if $(i)$ holds, or $\widetilde{P}_{\mu, V}(u_n)\to 0$ if $(ii)$ holds. Then $\{u_n\}$ is bounded in $H^1(\mathbb{R}^N)$.
\end{lemma}
\begin{proof}
In order to simplify the symbols in this proof, we write
\[
\begin{aligned}
&a_n:=\|\nabla u_n\|_2^2,\;\; b_n:=\int_{\mathbb{R}^N}V(x)u^2_n \ud x,\;\; c_n:=\int_{\mathbb{R}^N}V(x)u_n\nabla u_n\cdot x \ud x,\\
&d_n:=\|u_n\|_p^p,\;\; e_n:=\|u_n\|_{2^*}^{2^*},\;\; f_n:=\int_{\mathbb{R}^N}Wu_n^2 \ud x.
\end{aligned}
\]
{\bf Case (i)}. According to the assumptions of this lemma, we have
\begin{equation}\label{2207202352}
    a_n+b_n-\frac{2\mu}{p}d_n-\frac{2}{2^*}e_n=2{m}_{V,\, a}+o_n(1),
\end{equation}
\begin{equation}\label{22072023}
a_n+\frac{N}{2}b_n+c_n-\mu\gamma_p d_n -e_n=o_n(1).
\end{equation}\par
When $p\in (\hat{p}, 2^*)$, it follows from \eqref{22072023} that
\begin{equation}\label{2207202353}
\frac{2\mu}{p}d_n+\frac{2}{2^*}e_n\le \frac{2}{p\gamma_p}\left(\mu\gamma_p d_n+e_n\right)=\frac{2}{p\gamma_p}\left(a_n+\frac{N}{2}b_n+c_n \right)+o_n(1).
\end{equation}
Combining \eqref{2207202352} with \eqref{2207202353}, we get
\[
\begin{aligned}
a_n+b_n-\frac{2}{p\gamma_p}\left(a_n+\frac{N}{2}b_n+c_n \right)
\le a_n+b_n-\frac{2\mu}{p}d_n-\frac{2}{2^*}e_n+o_n(1)
= 2{m}_{V,\, a}+o_n(1).
\end{aligned}
\]
By the assumptions $(V_1)$ and $(V_3)$, we know that  $0\le-b_n\le \sigma_1 a_n$ and $|c_n|\le {\sigma_3} a_n$. From the above computations, it follows that if $N\ge 4$ (so $N\ge p\gamma_p$),
\[
\begin{aligned}
2{m}_{V,\, a}+o_n(1)\ge& (1-\frac{2}{p\gamma_p})a_n+(1-\frac{N}{p\gamma_p})b_n-\frac{2}{p\gamma_p}c_n
\\
\ge &(1-\frac{2}{p\gamma_p})a_n-\frac{2}{p\gamma_p} {\sigma_3} a_n\\
=&\left(1-\frac{2}{p\gamma_p}-\frac{2}{p\gamma_p} {\sigma_3}\right)a_n.
\end{aligned}
\]
Hence $\{a_n\}$ is bounded by \eqref{2207232113}. \par
If $N=3$ (hence $N< p\gamma_p$),
\[
\begin{aligned}
2{m}_{V,\, a}+o_n(1)
\ge &(1-\frac{2}{p\gamma_p})a_n-(1-\frac{3}{p\gamma_p}) \sigma_1 a_n-\frac{2}{p\gamma_p} {\sigma_3} a_n\\
=&\left(1-\frac{2}{p\gamma_p}-(1-\frac{3}{p\gamma_p}) \sigma_1-\frac{2}{p\gamma_p} {\sigma_3}\right)a_n.
\end{aligned}
\]
Hence $\{a_n\}$ is bounded by \eqref{22072321131}. \par
When $p=\hat{p}$, by \eqref{22072023}, we have
\begin{equation}\label{2301311129}
e_n=a_n+\frac{N}{2}b_n+c_n-\mu\gamma_{\hat{p}} d_n+o_n(1).
\end{equation}
Combining \eqref{2207202352} with \eqref{2301311129}, we get
\[
\begin{aligned}
2{m}_{V,\, a}+o_n(1)= \left(1-\frac2{2^*}\right)\left(a_n-\mu\gamma_{\hat{p}}d_n\right)-\frac2{2^*} c_n+\left(1-\frac{N}{2^*}\right)b_n+o_n(1).
\end{aligned}
\]\par
If $N\ge 4$ (we get $N\ge 2^*$), then
\begin{equation}\label{2301311130}
2{m}_{V,\, a}+o_n(1)\ge  \left(1-\frac2{2^*}\right)\left(1-\mu \gamma_{\hat{p}}C^{\hat{p}}_{N,\, {\hat{p}}}a^{\frac4N}\right)a_n-\frac2{2^*} {\sigma_3} a_n+o_n(1).
\end{equation}
Hence $\{a_n\}$ is bounded by \eqref{22072321132}.\par
If $N=3$ (so $N< 2^*$), we obtain from $d_n\le C^{\hat{p}}_{3,\, {\hat{p}}}a_n^{{\hat{p}}\gamma_{\hat{p}}/2}a^{\frac43}$ and $b_n\le 0$ that
\[
2{m}_{V,\, a}+o_n(1)\ge  \left(1-\frac2{6}\right)\left(1-\mu \gamma_pC^{\hat{p}}_{3,\, {\hat{p}}}a^{\frac43}\right)a_n-\frac2{6} {\sigma_3} a_n-\left(1-\frac{3}{6}\right)\sigma_1a_n+o_n(1).
\]
Hence $\{a_n\}$ is bounded by \eqref{22072321133}.

\bigskip
\noindent
{\bf Case (ii)}. By a similar argument as above, \eqref{2207202352} holds and \eqref{22072023} is replaced by
\begin{equation}\label{22072300}
a_n-f_n-\mu\gamma_p d_n -e_n=o_n(1).
\end{equation}\par
When $p\in (\hat{p}, 2^*)$, it follows from \eqref{2207202352} that
\begin{equation}\label{2207230038}
\frac{2\mu}{p}d_n+\frac{2}{2^*}e_n\le \frac{2}{p\gamma_p}\left(\mu\gamma_p d_n+e_n\right)=\frac{2}{p\gamma_p}\left(a_n-f_n\right)+o_n(1).
\end{equation}
Combining \eqref{2207202352} with \eqref{2207230038}, we get
\[
\begin{aligned}
a_n+b_n-\frac{2}{p\gamma_p}\left(a_n-f_n \right)
\le &a_n+b_n-\frac{2\mu}{p}d_n-\frac{2}{2^*}e_n+o_n(1) =2{m}_{V,\, a}+o_n(1).
\end{aligned}
\]
By $(V_1)$ and $(V_2)$, we have $|b_n|\le \sigma_1 a_n$ and $|f_n|\le \sigma_2 a_n$, then
\[
\begin{aligned}
2{m}_{V,\, a}+o_n(1)\ge& \left(1-\frac{2}{p\gamma_p}\right)a_n+b_n+\frac{2}{p\gamma_p}f_n
\\
\ge &\left(1-\frac{2}{p\gamma_p}\right)a_n- 2\sigma_2 a_n\\
=&\left(1-\frac{2}{p\gamma_p}-2\sigma_2\right)a_n.
\end{aligned}
\]
Hence $\{a_n\}$ is bounded by \eqref{2207232114}. \par
When $p=\hat{p}$, it follows from \eqref{22072300} that
\begin{equation}\label{2301311134}
e_n=a_n-f_n-\mu\gamma_{\hat{p}} d_n+o_n(1).
\end{equation}
Taking \eqref{2301311134} into \eqref{2207202352}, we obtain
\[
\begin{aligned}
2{m}_{V,\, a}+o_n(1)=&\left(1-\frac2{2^*}\right)\left(a_n-\mu\gamma_{\hat{p}} d_n\right)+b_n+\frac{2}{2^*}f_n +o_n(1)\\
\ge & \left(1-\frac2{2^*}\right)\left(1-\mu \gamma_{\hat{p}}C^{\hat{p}}_{N,\, {\hat{p}}}a^{\frac4N}\right)a_n-2\sigma_2a_n.
\end{aligned}
\]
Hence $\{a_n\}$ is bounded by \eqref{22072321141}.
\end{proof}
\section{Proof of main theorems}\label{2306251340}
In previous section, we have obtained a bounded Palais-Smale sequence $\{u_n\}$ at level $m_{V,\, a}$, which approach the Pohozaev manifold $\mathcal{P}_{\mu, V}$ or $\widetilde{\mathcal{P}}_{\mu, V}$. To finish our theorems, a key step is to prove that $\{u_n\}$ has a convergent subsequence in $H^1(\mathbb{R}^N)$. We can not work in the radial space $H_{rad}^1(\mathbb{R}^N)$ due to the fact that the potential function is not assumed to be radial, so we need to deal with the problem that $H^1(\mathbb{R}^N)\subset L^p(\mathbb{R}^N)$ is not compact. The Sobolev critical exponent brings another difficulties. The strategy in this section is using Lions' lemma and the estimate $m_{V,\, a}<\frac{1}{N}\mathcal{S}^{\frac{N}{2}}$ to get the following splitting result.
\begin{lemma}\label{2207110022}
Let $N\ge 3$, $p\in [\hat{p}, 2^*)$. Suppose that $V\in C^1(\mathbb{R}^N)$ satisfies
\begin{itemize}
\item[$(i)$] either $(V_1)$, $(\widetilde{V}_2)$, $(V_3)$, $(V_4)$ hold with \eqref{2301301540}, \eqref{22072321132}-\eqref{22072321131};
\item[$(ii)$] or $(V_1)$, $(V_2)$, $(V_4)$ hold with \eqref{2301301540}, \eqref{22072321141}, \eqref{2207232114}.
\end{itemize}
Let $\{u_n\}$ be a Palais-Smale sequence of $J_{\mu, V}$ at ${m}_{V,\, a}$ such that $P_{\mu, V}(u_n)\to 0$ if $(i)$ holds, or $\widetilde{P}_{\mu, V}(u_n)\to 0$ if $(ii)$ holds. In the sense of subsequence, we can assume that for some $u\in H^1(\mathbb{R}^N)$, $u_n \rightharpoonup u$ weakly in $H^1(\mathbb{R}^N)$. If this convergence is not strong, then there exist $k$ ( for some $k\in \mathbb{N}^+$) nontrivial solutions $w^1,\dots, w^k$ to \eqref{22071100} for some $\lambda<0$ and $k$ sequences $\{y_n^j\}\subset\mathbb{R}^N$ with $1\le j \le k$ such that $|y_n^j|\rightarrow \infty$, $|y_n^j-y_n^i|\to \infty$ for $i\neq j$ as $n\to \infty$, and
\begin{equation}\label{2207110010}
  u_n=u+\sum_{1\le j\le k} w^j(\cdot + y_n^j)+o_n(1) \;\; \text{in} \;\; H^1(\mathbb{R}^N).
\end{equation}
Moreover, we have
\begin{equation}\label{2207110009}
    \begin{aligned}
        \|u_n\|_2^2=\|u\|_2^2+\sum_{1\leq j\leq k} \|w^j\|_2^2,\;\;J_{\mu, V}(u_n)&=J_{\mu, V}(u)+\sum_{1\le j\le k} J_{\mu}(w^j)+o_n(1).
    \end{aligned}
\end{equation}
\end{lemma}
\begin{proof}
In this proof, we argue up to suitable subsequences.
Since $\{u_n\}$ is a Palais-Smale sequence of $J_{\mu, V}|_{S_a}$ at level ${m}_{V,\, a}$, there exists a sequence $\{\lambda_n\}\subset \mathbb{R}$ such that
\begin{equation}\label{22071020}
    -\Delta u_n+V(x)u_n-\mu |u_n|^{p-2}u_n-|u_n|^{2^{*}-2}u_n-\lambda_n u_n\rightarrow 0 ~\text{in} ~H^{-1}(\mathbb{R}^N),
\end{equation}
so
\begin{equation}
\lambda_n a^2=\|\nabla u_n\|_2^2+\int_{\mathbb{R}^N} V(x)u_n^2 \ud x-\mu \|u_n\|_p^p-\|u_n\|_{2^*}^{2^*}+o_n(1).
\end{equation}
It follows from Lemma \ref{22072315} that $\{u_n\}$ is bounded in $H^1(\mathbb{R}^N)$, so $\{\lambda_n\}$ is also a bounded sequence. We can assume $\lambda_n\rightarrow \lambda \in \mathbb{R}$. Because $u_n \rightharpoonup u$ weakly in $H^1(\mathbb{R}^N)$, $u$ is a solution to \eqref{22062610}.
Therefore, Lemma \ref{22062613} leads to either $\lambda<0$ or $u=0$. 

\medskip\noindent
{\bf Case (i)}. Suppose that $u=0$. Because $\{u^2_n\}$ is bounded in $L^{\frac{N}{N-2}}(\mathbb{R}^N)$ and $u^2_n(x)\rightarrow 0$ for a.e. $x\in \mathbb{R}^N$, we obtain that $u^2_n\rightharpoonup 0$ weakly in $L^{\frac{N}{N-2}}(\mathbb{R}^N)$. For any $R>0$, we put
 \begin{equation}
\int_{\mathbb{R}^N}V(x)u_n^2 \ud x =\int_{B_R(0)}V(x)u_n^2 \ud x + \int_{\mathbb{R}^N\backslash B_R(0)}V(x)u_n^2 \ud x.
\end{equation}
From $V\in L^{\frac{N}{2}}_{loc}(\mathbb{R}^N)$ and $\lim\limits_{|x|\rightarrow\infty} V(x)=0$ and taking $R$ large enough, it holds that
\begin{equation}\label{2307022355}
\int_{\mathbb{R}^N}V(x)u_n^2 \ud x=o_n(1).
\end{equation}
Similarly, it follows from $(V_3)$ that
\begin{equation}\label{2302032102}
\int_{\mathbb{R}^N}V(x)u_n\nabla u_n\cdot x \ud x=o_n(1).
\end{equation}
Hence
$${m}_{V,\, a}=J_{\mu, V}(u_n)+o_n(1)=\frac12 \|\nabla u_n\|_2^2-\frac{\mu}{p}\|u_n\|_p^p-\frac{1}{2^*}\|u_n\|_{2^*}^{2^*}+o_n(1),$$
and
\begin{equation}\label{23020321}
o_n(1)=P_{\mu, V}(u_n)=\|\nabla u_n\|_2^2-\mu\gamma_p\|u_n\|_p^p-\|u_n\|_{2^*}^{2^*}+o_n(1).
\end{equation}
So $P_{\mu}(u_n)\rightarrow 0$. If $\|\nabla u_n\|_2^2\to 0$, then $\|u_n\|_p^p\to 0$ and $\|u_n\|_{2^*}^{2^*}\to 0$. Hence ${m}_{V,\, a}=0$, that contradicts \eqref{2209120004}. As a result, up to a subsequence, we can assume that $\|\nabla u_n\|_2^2>c$ for some $c>0$.
 By Lemma \ref{22070919}, we conclude
$${m}_{V,\, a}=\underset{n\rightarrow \infty}{\lim}J_{\mu}(u_n)\ge m_a,$$
which contradicts Lemma \ref{2207091943}.\par
From the above, we know that $u\neq 0$ and $\lambda<0$. Letting $w^1_n:=u_n-u$, we obtain that $w^1_n\rightharpoonup 0$ weakly in $H^1(\mathbb{R}^N)$, $w^1_n\rightarrow 0$ strongly in $L^2_{loc}(\mathbb{R}^N)$, $L^p_{loc}(\mathbb{R}^N)$, and a.e. in $\mathbb{R}^N$. Set
\[
L:=\underset{n\rightarrow\infty}{\lim\inf}\underset{y\in \mathbb{R}^n}{\sup}\int_{B_1(y)}|w^1_n|^2 \ud x.
\]
If $L=0$, up to a subsequence, it follows from Lions' lemma (see \cite[Theorem 1.34]{Willem96}) that $\|w_n^1\|_p^p\rightarrow 0$. By Brezis-Lieb lemma (see \cite[Lemma 1.32]{Willem96}), we obtain
\[
\begin{aligned}
{m}_{V,\, a} &=J_{\mu, V}(u_n)+o_n(1)\\
&=J_{\mu, V}(u)+J_{\mu}(w_n^1)+o_n(1)\\
&=J_{\mu, V}(u)+\frac12 \|\nabla w^1_n\|_2^2-\frac{1}{2^*}\|w_n^1\|_{2^*}^{2^*}+o_n(1).
\end{aligned}
\]
In virtue of the boundedness of $\{w^1_n\}$ in $H^1(\mathbb{R}^N)$, we can assume that $\|\nabla w^1_n\|_2^2\rightarrow l\in\mathbb{R}$. By $P_{\mu, V}(u)=0$, $P_{\mu, V}(u_n)=o_n(1)$, Brezis-Lieb lemma and similar arguments in \eqref{2307022355}, \eqref{2302032102}, we get
\begin{equation*}
\begin{aligned}
  o_n&(1)=P_{\mu, V}(u_n)-P_{\mu, V}(u)\\
  &=\|\nabla w_n^1\|_2^2
+\frac{N}{2}\int_{\mathbb{R}^N}V(x)(w_n^1)^2 \ud x +\int_{\mathbb{R}^N}V(x)w_n^1\nabla w_n^1\cdot x \ud x-\mu\gamma_p\|w^1_n\|_p^p
-\|w^1_n\|_{2^*}^{2^*}+o_n(1)\\
&=\|\nabla w_n^1\|_2^2-\|w^1_n\|_{2^*}^{2^*}+o_n(1).
\end{aligned}
\end{equation*}
Thus
\begin{equation*}
l=\|\nabla w^1_n\|_2^2+o_n(1)\ge \mathcal{S}\|w^1_n\|_{2^*}^{2}+o_n(1)=\mathcal{S}l^{\frac{2}{2^*}}+o_n(1),
\end{equation*}
which implies either $l=0$ or $l\ge\mathcal{S}^{\frac{N}{2}}$.
However, if $l\ge \mathcal{S}^{\frac{N}{2}}$, using Lemma \ref{22072415}, then
$${m}_{V,\, a}=J_{\mu, V}(u)+\frac12 \|\nabla w^1_n\|_2^2-\frac{1}{2^*}\|w^1_n\|_{2^*}^{2^*}+o_n(1)\ge \frac{1}{N}\mathcal{S}^{\frac{N}{2}}+o_n(1),$$ which contradicts the fact that ${m}_{V,\, a} <m_a<\frac{1}{N}\mathcal{S}^{\frac{N}{2}}$ by Lemma \ref{2301312159} and Lemma \ref{2207091943}. Consequently, $l=0$ and we deduce that $\|\nabla w^1_n\|_2\rightarrow 0$. Multiplying \eqref{22071020}, \eqref{22062610} by $w^1_n$, we have
\[
\begin{aligned}
\|\nabla w^1_n\|_2^2+\int_{\mathbb{R}^N}V(x)|w^1_n|^2 \ud x=&\int_{\mathbb{R}^N}(\lambda_n u_n- \lambda u)w^1_n \ud x+\mu\int_{\mathbb{R}^N}(|u_n|^{p-2}u_n- |u|^{p-2}u)w^1_n \ud x\\
&+\int_{\mathbb{R}^N}(|u_n|^{2^*-2}u_n- |u|^{2^*-2}u)w^1_n \ud x+o_n(1).
\end{aligned}
\]
By $\|\nabla w^1_n\|_2\rightarrow 0$, we obtain
\[
\begin{aligned}
\lambda \int_{\mathbb{R}^N}(u_n-u)^2 \ud x \rightarrow 0.
\end{aligned}
\]
From $\lambda<0$, it follows that $u_n\rightarrow u$ in $H^1(\mathbb{R}^N)$.
A contradiction gives $L>0$. As a consequence, there exists a sequence $\{y_n^1\}\subset \mathbb{R}^N$ such that $|y_n^1|\to \infty$ and
\[
\int_{B_1(y^1_n)}|w^1_n|^2 \ud x \ge \frac{L}{2}.
\]
Letting $u^1_n(\cdot):=w^1_n(\cdot+y^1_n)$, there exists some $w^1\in H^1(\mathbb{R}^N)\backslash \{0\}$
such that $u^1_n\rightharpoonup w^1$ weakly in $H^1(\mathbb{R}^N)$ and $u^1_n(x)\rightharpoonup w^1(x)$ a.e. in $\mathbb{R}^N$. Setting $w^2_n:=u^1_n-w^1$, by Brezis-Lieb lemma again, we have
\[
\begin{aligned}
{m}_{V,\, a} &=J_{\mu, V}(u_n)+o_n(1)\\
&=J_{\mu, V}(u)+J_{\mu}(w_n^1)+o_n(1)\\
&=J_{\mu, V}(u)+J_{\mu}(w^1)+J_{\mu}(w^2_n)+o_n(1).
\end{aligned}
\]
It is clear that there exists some $C>0$ such that $|J_{\mu}(v)|\le C$ for any
$$v\in \{v\in H^1(\mathbb{R}^N): \|v\|_2\le a,\, \|\nabla v\|_2\le \underset{n}{\sup} \|\nabla u_n\|_2\}.$$
Using Lemma \ref{2207092032}, we get $J_{\mu}(w^j)\ge m_a$. By induction, we deduce that this process will terminate after a finite of times. Consequently, there exists some $k\in \mathbb{N}^+$ and $k$ nontrivial solutions $w^1,\dots, w^k$ to \eqref{22071100} and $k$ sequences $\{y_n^j\}\subset\mathbb{R}^N$ such that \eqref{2207110010}, \eqref{2207110009} hold.

\medskip\noindent
{\bf Case (ii)}. We only point out some differences to the above proof. First, by $(V_2)$, \eqref{2302032102} can be replaced by
\begin{equation*}
\int_{\mathbb{R}^N}W(x)u_n^2 \ud x=o_n(1),
\end{equation*}
and therefore \eqref{23020321} becomes
\[
o_n(1)=\widetilde{P}_{\mu, V}(u_n)=\|\nabla u_n\|_2^2-\mu\gamma_p\|u_n\|_p^p-\|u_n\|_{2^*}^{2^*}+o_n(1).
\]
Replace $P_{\mu, V}$ by $\widetilde{P}_{\mu, V}$. Similarly as the proof of Case (i), one can complete the proof of Case (ii).
\end{proof}

\begin{proof}[Proof of Theorem \ref{22071658} completed]
According to the assumptions of Theorem \ref{22071658}, when $p=\hat{p}$, if the mass $a$ is small enough, then the conditions of Lemma \ref{22072415} (i), Lemma \ref{2207091943}, Lemma \ref{22072020} (i), Lemma \ref{22072315} (i) and Lemma \ref{2207110022} (i) will be satisfied. While for $p\in (\hat{p}, 2^*)$, the mass $a$ is not required to be small. \par
By Lemma \ref{22072020}, there exists a Palais-Smale sequence $\{u_n\}$ of $J_{\mu, V}|_{S_a}$ at level ${m}_{V,\, a}$, which satisfies \eqref{22072019}. Lemma \ref{22072315} leads to the boundedness of $\{u_n\}$ in $H^1(\mathbb{R}^N)$. Up to a subsequence, we can assume that there is some $u\in H^1 (\mathbb{R}^N)$ such that $u_n \rightharpoonup u$ weakly in $H^1(\mathbb{R}^N)$, and a.e. in $\mathbb{R}^N$. If this convergence is not strong, then by Lemma \ref{2207110022}, there exist $\lambda<0,\, k\in \mathbb{N}^+$ and $k$ nontrivial solutions $w^1,\dots, w^k$ of \eqref{22071100} such that \eqref{2207110009} holds. By Lemma \ref{22072415}, we have $J_{\mu, V}(u)\ge 0$. Furthermore, it follows from Lemma \ref{2207092032} and Lemma \ref{2207091943} that $J_{\mu}(w^j)\ge m_a >m_{V,\, a}$, which contradicts \eqref{2207110009}. Thus $u_n\rightarrow u$ strongly in $H^1(\mathbb{R}^N)$.\par
In the sequel, we shall prove $u>0$. In virtue of $V\in C^1(\mathbb{R}^N)$,  we can obtain from Remark \ref{2302161844} that $u$ is a classical solution. Due to $\|u_n^-\|_2\to 0$, it follows that $u\ge 0$. By the strong maximum principle, we get $u>0$ in $\mathbb{R}^N$.
\end{proof}
\begin{proof}[Proof of Theorem \ref{2207241248} completed] This proof is very similar to the proof of Theorem \ref{22071658} where we need to use Lemma \ref{22072415} instead of Lemma \ref{2207241537}.
\end{proof}

\appendix\section{Appendix}
Here we will show the differential property for the potential term of $J_{\mu,\, V}$. In addition, we will give some Pohozaev identities for \eqref{22062610}, that is, any $H^1$ solution to \eqref{22062610} will satisfy \eqref{2302071001}. 

\begin{lemma}\label{2301272023}
Assume that $N\ge 1$. We have the following assertions:
\begin{itemize}
\item[$(i)$] If the potential $V$ satisfies \eqref{2301271916}, then for $u\in H^1(\mathbb{R}^N)$,
\begin{equation}\label{23012722}
\frac{d}{dt} \int_{\mathbb{R}^N} V(x)u^2(tx) \ud x\Big|_{t=1}=2\int_{\mathbb{R}^N} V(x)u(x)\nabla u(x)\cdot x \, \ud x.
\end{equation}\par
\item[$(ii)$] If the potential $V\in C^1(\mathbb{R}^N)$ satisfies \eqref{2301262110} and \eqref{2301271518}, then for $u\in H^1(\mathbb{R}^N)$,
\begin{equation}\label{2301272258}
\frac{d}{dt} \int_{\mathbb{R}^N} V\left(\frac{x}{t}\right)u^2 \ud x\Big|_{t=1}=-2\int_{\mathbb{R}^N} W(x)u^2 \ud x.
\end{equation}
\end{itemize}
\end{lemma}

\begin{proof}
We only prove (i), because (ii) is similar. From Remak \ref{2304061410}, \eqref{2301262110} holds and $\sigma_1$ can be chosen to be $\frac{2}{N-2}\sigma_3$. If $u\in C_c^{\infty}(\mathbb{R}^N)$, there holds
\begin{equation}\label{23012718}
\begin{aligned}
\quad \int_{\mathbb{R}^N} V(x)u^2(tx) \ud x-\int_{\mathbb{R}^N} V(x)u^2(x) \ud x & = \int_{\mathbb{R}^N}\int_1^t 2V(x)u(sx)\nabla u(sx)\cdot x \,\ud s \ud x\\
& = \int_1^t\int_{\mathbb{R}^N} 2V(x)u(sx)\nabla u(sx)\cdot x \, \ud x\ud s.\\
\end{aligned}
\end{equation}
Since $C_c^{\infty}(\mathbb{R}^N)$ is dense in $H^1(\mathbb{R}^N)$, for $u\in H^1(\mathbb{R}^N)$, there exists a sequence $\{u_n\}\subset C_c^{\infty}(\mathbb{R}^N)$ such that $u_n\to u$ in $H^1(\mathbb{R}^N)$. By \eqref{2301262110} and H\"{o}lder inequality, we get
\[
\begin{aligned}
\left|\int_{\mathbb{R}^N} V(x)u_n^2 \ud x-\int_{\mathbb{R}^N} V(x)u^2\ud x\right| &\le \left|\int_{\mathbb{R}^N} V(x)(u_n-u)^2\ud x\right|+2\left|\int_{\mathbb{R}^N} V(x)(u_n-u)u\ud x\right|\\
& \le C\sigma_3\|\nabla (u_n-u)\|_2^2+2C\sigma_3\|\nabla (u_n-u)\|_2\|\nabla u\|_2.\\
\end{aligned}
\]
On the other hand, for $s\in (0, \infty)$, by \eqref{2301271916} and H\"{o}lder inequality, we have
\[
\begin{aligned}
&\quad \left|\int_{\mathbb{R}^N} V(x)u_n(sx)\nabla u_n(sx)\cdot x \ud x-\int_{\mathbb{R}^N} V(x)u(sx)\nabla u(sx)\cdot x \, \ud x\right|\\
& \le \left|\int_{\mathbb{R}^N} V(x)u_n(sx)(\nabla u_n(sx)-\nabla u(sx))\cdot x \, \ud x\right|+\left|\int_{\mathbb{R}^N} V(x)(u_n(sx)-u(sx))\nabla u(sx)\cdot x \, \ud x\right|\\
& \le {\sigma_3} s^{1-{N}}\|\nabla (u_n-u)\|_2\|\nabla u_n\|_2+{\sigma_3} s^{1-{N}}\|\nabla (u_n-u)\|_2\|\nabla u\|_2.
\end{aligned}
\]
Thus, \eqref{23012718} holds for $u\in H^1(\mathbb{R}^N)$ and $t\in (0, \infty)$. Next, we claim that
\begin{equation}\label{2301271802}
g(s):=\int_{\mathbb{R}^N} 2V(x)u(sx)\nabla u(sx)\cdot x \, \ud x
\end{equation}
is continuous for $s\in (0, \infty)$. Without loss of generality we only prove the continuity of $g$ at $s=1$. It is seen from \eqref{2301271916} that
\[
\begin{aligned}
|g(s)-g(1)|&\le 2\int_{\mathbb{R}^N} \left|V(x)u(sx)\nabla u(sx)\cdot x - V(x)u(x)\nabla u(x)\cdot x\right| \, \ud x\\
& \le 2\int_{\mathbb{R}^N} \left|V(x)u(sx)\nabla u(sx)\cdot x - V(x)u(sx)\nabla u(x)\cdot x\right| \, \ud x\\
&\quad +2\int_{\mathbb{R}^N} \left|V(x)u(sx)\nabla u(x)\cdot x - V(x)u(x)\nabla u(x)\cdot x\right| \, \ud x\\
& \le 2  \| \nabla u(sx)-\nabla u(x) \|_2  \| V(x)|x|u(sx) \|_2+2  \|\nabla u \|_2  \| V(x)|x|(u(sx)-u(x)) \|_2\\
& \le 2{\sigma_3} s^{{1-\frac{N}{2}}} \| \nabla u(sx)-\nabla u(x) \|_2 \|\nabla u\|_2+2  \|\nabla u \|_2  \| \nabla (u(sx)-u(x)) \|_2.\\
\end{aligned}
\]
It is clear from Brezis-Lieb lemma (see \cite[Lemma 1.32]{Willem96}) that
\[
\underset{s\to 1}{\lim} \| \nabla u(sx)-\nabla u(x) \|_2 = \underset{s\to 1}{\lim} \| \nabla(u(sx)- u(x)) \|_2= 0,
\]
So we get the claim. Finally, by \eqref{23012718} and \eqref{2301271802},
\[
\frac{d}{dt} \int_{\mathbb{R}^N} V(x)u^2(tx) \ud x\Big|_{t=1} =\underset{t\to 1}{\lim}\frac{1}{t-1}\int_1^t g(s) \ud s=g(1).
\]
The proof is completed. \end{proof}

Next, we prove the Pohozaev identities.
\begin{proposition}\label{2302171335}
Let $N\ge 3$, $p\in (2, 2^*)$. Then \eqref{2302071001} holds true for any $u \in H^1({\mathbb R}^N)$ solution to \eqref{22062610}.
\end{proposition}
\begin{proof}
Suppose that $u$ solves \eqref{22062610}. Use $u$ as a test function to \eqref{22062610}, there holds
\begin{equation}\label{2302171325}
\|\nabla u\|_2^2+\int_{\mathbb{R}^N}V(x)u^2 \ud x=\lambda \|u\|_2^2+\mu \|u\|_p^p+\|u\|_{2^*}^{2^*}.
\end{equation}
By \cite[Proposition 2.1]{LM2014}, we know that if $V\in C^1(\mathbb{R}^N)$ and \eqref{2301262110}, \eqref{2301271518} are satisfied, any solution $u \in H^1({\mathbb R}^N)$ to \eqref{22062610} will satisfy $\widetilde{P}_{\mu, V}(u)=0$. 

Now, it suffices to check that if \eqref{2301271916} is valid, any solution to \eqref{22062610} satisfies $P_{\mu, V}(u)=0$. We follow the proof of \cite[Proposition 2.1]{LM2014} and \cite[Theorem B.3]{Willem96}. Let $\eta\in C_c^{\infty}(\mathbb{R})$ be a cut-off function satisfying $0\le \eta \le 1$, $\eta(r)=1$ for $r\le 1$ and $\eta(r)=0$ for $r\ge 2$. Let
\[
\eta_n(x)=\eta\left(\frac{|x|^2}{n^2}\right), \quad \forall\; n \ge 1.
\]
Thus, $0 \le \eta_n\le 1$ and there exists $c>0$ such that
\[
|x||\nabla \eta_n(x)|\le c, \quad \forall \; x \in {\mathbb R}^N, n \ge 1.
\]
By \cite[B.3 Lemma]{Struwe2008}, $u\in W_{loc}^{2, q}(\mathbb{R}^N)\cap H^1(\mathbb{R}^N)$ for any $q\ge 1$, since $V\in L_{loc}^{\frac{N}{2}}(\mathbb{R}^N)$. Hence
\begin{equation}\label{2302171153}
\begin{aligned}
\eta_n\Delta u(x\cdot \nabla u) & = {\rm div}\left[(x\cdot \nabla u)\eta_n\nabla u - \frac{\eta_n|\nabla u|^2}{2}x\right]\\
&\quad - (x\cdot \nabla u)(\nabla u\cdot\nabla \eta_n) + \frac{N-2}{2}\eta_n|\nabla u|^2+ \frac{|\nabla u|^2}{2}\nabla \eta_n \cdot x.
\end{aligned}
\end{equation}
Let
\[
g(x, u):=-V(x)u+f(u),
\]
with $f(u):=\lambda u+\mu |u|^{p-2}u+|u|^{2^{*}-2}u$. Therefore
\begin{equation}\label{2302171154}
g(x, u)\eta_nx\cdot \nabla u = -V(x)u\eta_nx\cdot \nabla u+ {\rm div}\big[\eta_n xF(u)\big] -N\eta_nF(u)- F(u)\nabla \eta_n\cdot x,
\end{equation}
where
\[
F(u)=\int_0^uf(t) \ud t=\frac{\lambda}{2}|u|_2^2+\frac{\mu}{p} |u|_p^p+\frac{1}{2^*}|u|_{2^*}^{2^*}.
\]
By \eqref{2302171153} and \eqref{2302171154}, under \eqref{2301271916}, applying divergence theorem and Lebesgue's dominated convergence theorem, taking $n$ tend to $\infty$, we arrive at
\begin{equation}\label{2302171324}
\frac{N-2}{2}\|\nabla u\|_2^2=\int_{\mathbb{R}^N}V(x)ux\cdot \nabla u \ud x+\frac{N\lambda}{2}\|u\|_2^2+\frac{N\mu}{p} \|u\|_p^p+\frac{N}{2^*}\|u\|_{2^*}^{2^*}.
\end{equation}
By \eqref{2302171325} and \eqref{2302171324}, we get $P_{\mu, V}(u)=0$.
\end{proof}


\begin{thebibliography}{10}
\begingroup
\setstretch{1}
\bibitem{CJ2022} C.O. Alves and C. Ji, Normalized solutions for the Schr\"{o}dinger equations with $L^2$-subcritical growth and different types of potentials, {\it J. Geom. Anal.} {\bf 32} (2022), Paper No. 165, 25 pp.

\bibitem{Ambrosetti2006} A. Ambrosetti and A. Malchiodi, Perturbation Methods and Semilinear Elliptic Problems on $\mathbb{R}^N$, Progress in Mathematics, vol.240, Birkh\"{a}user Verlag, Basel, 2006.

\bibitem{BMRV2021} T. Bartsch, R. Molle, G. Riey and G. Verzini, Normalized solutions of mass supercritical Schrödinger equations with potential, {\it Comm. Partial Differential Equations} {\bf 46}  (2021), 1729-1756.

\bibitem{BC1990} V. Benci and G. Cerami, Existence of positive solutions of the equation $-\Delta u+a(x)u=u^{\frac{N+2}{N-2}}$ in $\mathbb{R}^N$, {\it J. Funct. Anal.} {\bf 88} (1990), 90-117.

\bibitem{Cerami2006} G. Cerami, Some nonlinear elliptic problems in unbounded domains, {\it Milan J. Math.} {\bf 74} (2006), 47-77.


\bibitem{DZ2022} Y. Ding and X. Zhong, Normalized solution to the Schr\"{o}dinger equation with potential and general nonlinear term: Mass super-critical case, {\it J. Differential Equations} {\bf 334} (2022), 194-215.

\bibitem{Ghoussoub1993} N. Ghoussoub, Duality and Perturbation Methods in Critical Point Theory, Cambridge University Press, 1993.

\bibitem{CGS1989} L. Caffarelli, B. Gidas and J. Spruck, Asymptotic symmetry and local behavior of semilinear elliptic equations with critical Sobolev growth, {\it Comm. Pure Appl. Math.} {\bf 42} (1989), 271-297.

\bibitem{IM2020} N. Ikoma and Y. Miyamoto, Stable standing waves of nonlinear Schr\"odinger equations with potentials and general nonlinearities,
{\it Calc. Var. Partial Differential Equations} {\bf 59} (2020), Paper No. 48, 20 pp.

\bibitem{Jeanjean97} L. Jeanjean, Existence of solutions with prescribed norm for semilinear elliptic equations, {\it Nonlinear Anal.} {\bf 28} (1997), 1633-1659.


\bibitem{JL2022} L. Jeanjean and T.T. Le,
Multiple normalized solutions for a Sobolev critical Schrödinger equation,
{\it Math. Ann.} {\bf 384} (2022), 101-134.

\bibitem{JL2020}L. Jeanjean and S. Lu, A mass supercritical problem revisited, {\it Calc. Var. Partial Differential Equations} {\bf 59} (2020), Paper No. 174, 43 pp.

\bibitem{Kong1989} M.K. Kwong, Uniqueness of positive solutions of $\Delta u-u + u^p = 0$ in $\mathbb{R}^N$, {\it Arch. Ration. Mech. Anal.} {\bf 3} (1989), 243-266.

\bibitem{LM2014} R. Lehrer and L. Maia, Positive solutions of asymptotically linear equations via Pohozaev manifold, {\it J. Funct. Anal.} {\bf 266} (2014), 213-246.

\bibitem{Lions84-1} P.L. Lions, The concentration-compactness principle in the calculus of variations. The locally compact case. I., {\it Ann. Inst. H. Poincar\'e Anal. Non Lin\'eaire} {\bf 1} (1984), 109-145.

\bibitem{Lions84-2} P.L. Lions, The concentration-compactness principle in the calculus of variations. The locally compact case. II., {\it Ann. Inst. H. Poincar\'e Anal. Non Lin\'eaire} {\bf 1} (1984), 223-283.

\bibitem{MRV2022} R. Molle, G. Riey and G. Verzini,  Normalized solutions to mass supercritical Schr\"{o}dinger equations with negative potential, {\it J. Differential Equations} {\bf 333} (2022), 302-331.

\bibitem{N1959} L. Nirenberg, On elliptic partial differential equations, {\it Ann. Scuola Norm. Sup. Pisa Cl. Sci.} {\bf 13} (1959), 115-162.

\bibitem{PPVV2021} B. Pellacci, A. Pistoia, G. Vaira and G. Verzini, Normalized concentrating solutions to nonlinear elliptic problems, {\it J. Differential Equations} {\bf 275} (2021), 882-919.

\bibitem{Soave20_JDE}N. Soave, Normalized ground states for the NLS equation with combined nonlinearities, {\it J. Differential Equations} {\bf 269} (2020), 6941-6987.

\bibitem{Soave20_JFA}N. Soave, Normalized ground states for the NLS equation with combined nonlinearities: the Sobolev critical case, {\it J. Funct. Anal.} {\bf 279} (2020), 108610, 43 pp.

\bibitem{Struwe2008} M. Struwe, Variational Methods. Applications to Nonlinear Partial Differential Equations and Hamiltonian Systems, fourth edition, Results in Mathematics and Related Areas. 3rd Series. A Series of Modern Surveys in Mathematics, vol.34, Springer-Verlag, Berlin, 2008.

\bibitem{Talenti1976} G. Talenti, Best constant in Sobolev inequality, {\it Ann. Mat. Pura Appl.} {\bf 110} (1976), 353-372.

\bibitem{WW22} J. Wei and Y. Wu,
Normalized solutions for Schrödinger equations with critical Sobolev exponent and mixed nonlinearities, {\it J. Funct. Anal.} {\bf283} (2022), Paper No. 109574, 46 pp.



\bibitem{Willem96} M. Willem, Minimax Theorems, Birkh\"{a}user, Boston, 1996.


\bibitem{ZZ2021} X. Zhong and W. Zou, A new deduction of the strict sub-additive inequality and its application: ground state normalized solution to Schr\"{o}dinger equations with potential, {\it Differential Integral Equations} {\bf 36} (2023), 133-160.
\endgroup
\end{thebibliography}
\end{document}